\documentclass[]{svjour3}

\RequirePackage{amsmath, amsfonts, amssymb}
\RequirePackage{graphicx, color}
\RequirePackage{bm}
\RequirePackage{hyperref}
\RequirePackage{pifont}
\RequirePackage[english]{babel}
\RequirePackage[per-mode=symbol]{siunitx}
\usepackage{tikz}

\smartqed  

\newcommand{\D}{\mathcal{D}}
\newcommand{\Dt}{\tilde{\mathcal{D}}}
\newcommand{\n}{\mathbf{n}}
\newcommand{\K}{\mathbf{K}}
\newcommand{\jump}[1]{\left[\!\!\left[#1\right]\!\!\right]}
\newcommand{\Hrm}{\mathrm{H}}
\newcommand{\Vrm}{\mathrm{V}}
\newcommand{\Lrm}{\mathrm{L}}
\newcommand{\Ucal}{\mathcal{U}}
\newcommand{\Hcal}{\mathcal{H}}
\newcommand{\Gmat}{{\bf{G}} }
\newcommand{\aD}{a_{\mathcal{D}}}
\newcommand{\ai}{a_{F_i}}
\newcommand{\TD}{\mathcal{T}^\D_{\delta_\D}}
\newcommand{\TFi}{\mathcal{T}^{i}_{\delta_{F_i}}}
\newcommand{\TGi}{\mathcal{T}^{i}_{\delta_{\Gamma_i}}}
\newcommand{\TS}{\mathcal{T}^{i}_{\delta_{S_m,i}}}
\newcommand{\ND}{\mathcal{N}_{h_\D}}
\newcommand{\NFi}{\mathcal{N}_{h_i}}
\newcommand{\NGi}{\mathcal{N}_{q_i}}
\newcommand{\NGt}{\mathcal{N}_{q}}
\newcommand{\NS}{\mathcal{N}_{u_i}^m}
\newcommand{\Jd}{\mathcal{J}}
\newcommand{\Cmat}{{\bf{C}}}
\newcommand{\Bmat}{{\bf{B}}}
\newcommand{\Nhf}{\mathcal{N}_{h_\Fcal}}
\newcommand{\Nht}{\mathcal{N}_h}
\newcommand{\NSt}{\mathcal{N}_u}
\newcommand{\NSip}{\mathcal{N}_{u_i}^+}
\newcommand{\NSi}{\mathcal{N}_{u_i}}
\newcommand{\Rmat}{{\bf{R}}}
\newcommand{\Amat}{{\bf{A}}}
\newcommand{\Dmat}{{\bf{D}}}
\newcommand{\Emat}{{\bf{E}}}
\newcommand{\Zmat}{{\bf{Z}}}
\newcommand{\Omat}{{\bf{O}}}
\newcommand{\Nw}{\mathcal{N}_w}
\newcommand{\Fcal}{\mathcal{F}}

\title{An optimization approach for flow simulations in poro-fractured media with complex geometries}

\author{Stefano Berrone \and Alessandro D'Auria \and Stefano Scial{\`o}}

\authorrunning{Berrone, D'Auria, Scial\`o} 

\titlerunning{Optimization approach for flow simulations in complex DFMs}

\institute{S. Berrone \and A. D'Auria \and S. Scial\`{o} \at
  Dipartimento di Scienze Matematiche, Politecnico di Torino, Corso
  Duca degli Abruzzi 24, 10129 Torino, Italy
  \\
  Members of the INdAM research group GNCS
  \\
  \email{stefano.berrone@polito.it - alessandro.dauria@polito.it -
    stefano.scialo@polito.it}}

\date{Received: date / Accepted: date}

\begin{document}\maketitle

%
%
%
%
%
%
%
%


\begin{abstract}
A new discretization approach is presented for the simulation of flow in complex poro-fractured media described by means of the Discrete Fracture and Matrix Model. The method is based on the numerical optimization of a properly defined cost-functional and allows to solve the problem without any constraint on mesh generation, thus overcoming one of the main complexities related to efficient and effective simulations in realistic DFMs.


\end{abstract}

\section{Introduction}

The present work deals with the simulation of the flow in the subsoil, modelled by means of the Discrete Fracture and Matrix (DFM) model. According to this model, underground fractures are represented as planar polygons arbitrarily oriented in a three dimensional porous matrix. The flows considered here are governed by the Darcy law in the three dimensional matrix and by an averaged Darcy law on each fracture plane, with suitable matching conditions at fracture-matrix interfaces and at fracture intersections. The quantity of interest is the hydraulic head, given by the sum of the pressure head and elevation. Single phase stationary flow is considered, with the assumption of continuity of the hydraulic head at both fracture-matrix interfaces and at fracture-fracture intersections and no longitudinal flow is allowed along fracture intersections. This is a simplified model with respect to other DFM approaches, described, for example in \cite{MJR2005} or, more recently, in \cite{Boon2018}, but still representative of realistic configurations, characterized, e.g., by highly permeable fractures. The main focus of the present work is on geometrical complexity aspects, proposing a problem formulation and a numerical approach suitable for complex and randomly generated networks.
The described approach can however be extended to different flow models and different coupling conditions.
The geometrical complexity of DFM models mainly arises from the multi-scale nature of the resulting domains and from the presence of multiple intersecting interfaces, where the solution displays an irregular behavior. DFM models are proposed as an alternative to homogenization techniques \cite{Qi2005}, dual and multy-porosity models \cite{DPbook}, or embedded discrete fracture matrix (EDFM) models \cite{Li2008,Moinfar2014,2015WR017729}, and are characterized by the explicit representation of the underground fractures, dimensionally reduced to planar interfaces into the porous matrix. As a consequence of the random orientation, fractures usually form an intricate system of intersections, with the presence of fractures with very different sizes spanning several orders of magnitude that generate intersections with huge geometrical complexities as, for example, 2D and 3D geometrical objects with very different dimension and objects with enormous aspect ratios. The research on effective numerical tools for DFM simulations is particularly active, see e.g. \cite{Angot2009,Ahmed2015,Brenner2016,Antonietti2016,BPSf,Odsaeter2019,Antonietti2019,Chernyshenko2020}. One of the key aspects is the meshing of the domain, with a mesh conforming to the interfaces, suitable for standard approaches for the imposition of interface conditions. The generation of a conforming mesh for realistic fracture networks might, in fact, result in an impossible task, for the extremely high number of geometrical constraints.  
The mesh conformity constraint at the interfaces can be relaxed by using extended finite elements as suggested, e.g., by \cite{FS13,formaggia2014}. Different approaches are based on the Mimetic Finite Difference method (MFD) \cite{Lipnikov2014}, as described, for example, in \cite{Wheeler2015,Antonietti2016}, or on Hybrid High Order (HHO) methods as proposed by \cite{Chave2018}, where a partial non-conformity is allowed between the mesh of the porous medium and of the fractures, or also on Discontinuous Galerkin discretizations, as in \cite{Antonietti2019}. Two or multi-point flux approximation based techniques are described in \cite{Sandve2012,Faille2016} and gradient schemes in \cite{Brenner2016}. Virtual Element (VEM) based discretizations have also been recently investigated to ease the mesh generation process in complex DFMs, as in \cite{BBFPSV} where the VEM is coupled to the Boundary Element method, and in \cite{Fumagalli2019}, in \cite{Coulet2019} for poro-elasticity problems, or in \cite{Benedetto2020} where an arbitrary order mixed VEM formulation is proposed.

This work presents a development of an optimization-based approach, first proposed for Discrete Fracture Networks \cite{BPSa,BPSb,BBoS,BSV,BBV2019} and recently extended to DFM problems in \cite{BPSf}. This approach avoids any mesh conformity requirement for the imposition of interface conditions, which are instead enforced through the minimization of a properly defined cost functional. The computation of the quantities involved in functional definition does not require any constraint on the mesh. Further, the resolution of the optimization problem via a gradient-based scheme allows to de-couple the problems on each fracture and the problem on the porous matrix, thus paving the way for an efficient parallel implementation of the numerical scheme, similarly to what done in \cite{BSV,BDV}. The discretization scheme described in \cite{BPSf} relies on the Boundary Element Method for the discretization of the problem on three dimensional matrix blocks, thus requiring the splitting of the original three dimensional domain into sub-domains not crossing the fractures, and thus implying a partial mesh conformity at the fracture-matrix interfaces. Here, the three dimensional domain is not split into sub-domains and Finite Elements are used for the discretization of the matrix, on tetrahedral elements that can arbitrarily cross the fractures. Finite elements on triangular meshes are used for the fractures, with elements not conforming to the tetrahedral mesh and also arbitrarily placed with respect to fracture-fracture intersections. The proposed discretization approach thus greatly improves the usability of the method to general DFM geometries, allowing a trivial meshing process of extremely complex domains, thanks to the complete independence of the mesh from all the interfaces.

The structure of the manuscript is the following: Section~\ref{ProbDescr} describes both the classical and the optimization based formulation of the flow problem in a DFM; the following Section~\ref{DiscrProb} describes the derivation of the discrete problem and the proof of its well posedness; Section~\ref{UncProb} shows how an equivalent unconstrained optimization problem is derived, and the gradient based scheme used for problem resolution; Section~\ref{NumRes} reports some numerical results and finally some conclusions are proposed in Section~\ref{Conc}.

\section{Problem description}
\label{ProbDescr}

\begin{figure}
\centering
\includegraphics[width=0.7\textwidth]{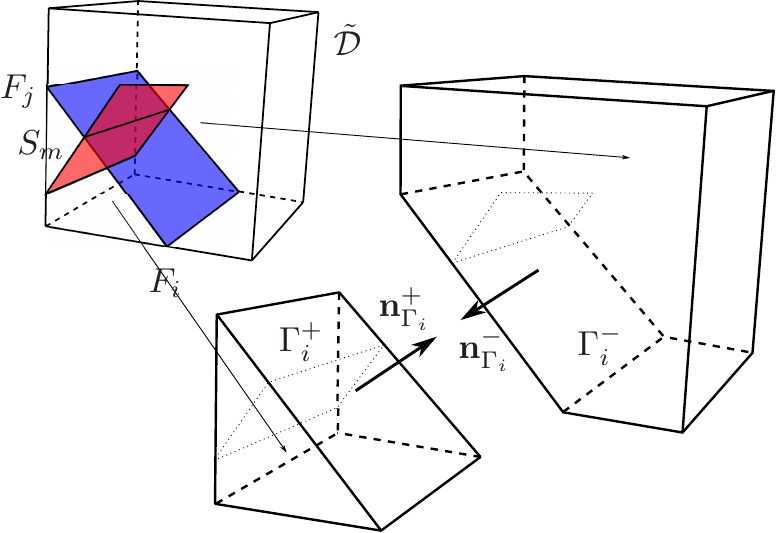}
\caption{Nomenclature exemplification}
\label{block_split}
\end{figure}

This section is devoted to a brief description of the problem of interest, referring to \cite{BPSf} for a more detailed exposition and well posedness results. Let us consider a polyhedral block of porous material, denoted as $\D$, crossed by a fracture network $\Omega$ given by the union of planar polygonal fractures $F_i$, $i=1,\ldots,N_F$ in the three-dimensional space, i.e. $\Omega=\bigcup_{i=1}^{N_F} F_i$. We further denote by $\Fcal$ the set of all fracture indexes. Fractures might intersect, and fracture intersections, also called traces, are indicated as $S_m$, $m=1,\ldots,N_S$. We assume, for simplicity, that each trace is given by the intersection of exactly two fractures, such that an injective map $\sigma: [1,\ldots,N_S]\mapsto [1,\ldots,N_F]\times [1,\ldots,N_F]$ can be defined between a trace index and a couple of fracture indexes, as $\sigma(m)=\{i,j\}$ being $S_m=\bar{F}_i\cap\bar{F}_j$. Further, $\mathcal{S}_i$ is the set of indexes of all the traces on fracture $F_i$ and $\mathcal{S}$ the set of indexes of all the traces in the network. Let us introduce the domain $\Dt=\D\setminus \bar{\Omega}$, thus given by the original block $\D$ without the internal fractures. Calling $\partial \Dt$ the boundary of $\Dt$, let us denote by $\Gamma_i^\pm$ the portion of $\partial \Dt$ that matches fracture $F_i$, for $i=1,\ldots,N_F$, the superscript ``$+$'' or ``$-$'' referring to one of the two sides of the boundary ``around'' the fracture (see Figure~\ref{block_split}); the unit normal vector to $\Gamma_i^\pm$ is $\n_{\Gamma_i}^\pm$, always pointing outward from $\Dt$. A jump operator is introduced for any sufficiently regular vector function $\mathbf{v}$ on $\Dt$, defined as the jump of $\mathbf{v}$ along the normal direction to the faces $\Gamma_i^\pm$:
\begin{displaymath}
\jump{\mathbf{v} \cdot \n}_{\Gamma_i}:= \left(\mathbf{v}_{|\Gamma_i^+} \cdot \n_{\Gamma_i^+}\right)-\left(\mathbf{v}_{|\Gamma_i^-} \cdot \n_{\Gamma_i^+}\right).
\end{displaymath}
Similarly, for $i=1,\ldots,N_F$ we denote by $\tilde{F}_i$ the fracture $F_i$ without traces, i.e. $\tilde{F}_i=F_i\setminus \bigcup_{m\in\mathcal{S}_i} S_m$, and for each trace $S_m$, $m\in\mathcal{S}_i$, for any sufficiently regular vector function $\mathbf{w}_i$ on $F_i$, the jump of the normal component of $\mathbf{w}_i$ across trace $S_m$ on $F_i$ is denoted as:
\begin{displaymath}
\jump{\mathbf{w}_i \cdot \n}_{S_m}:=\left(\mathbf{w}_{i|S_m^+} \cdot \n^i_{S_m}\right)-\left(\mathbf{w}_{i|S_m^-} \cdot \n^i_{S_m}\right),
\end{displaymath}
with $S_m^\pm$ the two sides of the portion of the boundary of $\tilde{F}_i$ lying on $S_m$ and $\n^i_{S_m}$ the normal unit vector to $S_m$ with a fixed orientation on $F_i$. These jump operators are easily extended to functions defined on the whole 3D domain $\D$ and on the whole fractures $F_i$, $i=1,\ldots,N_F$, with the $\pm$ superscripts still denoting the two sides of the interface $\Gamma_i\equiv F_i$, $\forall i=1,\ldots,N_F$, or $S_m$, $\forall m=1,\ldots,N_S$ .

The portion of $\partial\Dt$ not matching any fracture is split in a Dirichlet part $\Gamma_D$ and a Neumann part $\Gamma_N$, $\Gamma_D\cap\Gamma_N=\emptyset$, where, for simplicity of exposition, we assume homogeneous Dirichlet and Neumann boundary conditions are enforced. Similarly, the boundary of each fracture $\partial F_i$, $i=1,\ldots,N_F$, is split in a Dirichlet and Neumann part, $\gamma_{iD}$ and $\gamma_{iN}$, respectively. If fracture $F_i$ lies in the interior of $\D$, then we set $\gamma_{iD}=\emptyset$, and homogeneous Neumann boundary conditions are prescribed on $\gamma_{iN} \equiv \partial F_i$. If $N_F=1$, we assume that $|\gamma_{1D}|>0$, whereas, if there is more than one fracture in the network, we allow $\gamma_{iD}=\emptyset$ for $i=1,\ldots,N_F$.
The problem of the equilibrium distribution of the hydraulic head in $\D$ can be then stated in strong formulation as:
\begin{eqnarray}
-\nabla \cdot \left(\K_\D \nabla H_\D\right)&=&f \quad \text{in} \ \Dt \\
-\nabla_i \cdot \left(\K_i \nabla_i H_i\right) &=& -\jump{\K_\D \nabla H_\D \cdot \n}_{\Gamma_i} \quad \text{in} \ \tilde{F}_i, \ i=1,\ldots,N_F\\
(H_\D)_{|\Gamma_i^\pm}&=&H_i, \quad \ i=1,\ldots,N_F \label{MFcont} \\
H_i&=&H_j, \quad \ \text{on} \ S_m,\ m=1,\ldots,N_S,\ i,j=\sigma(m)\label{FFcont}\\
\jump{\K_i \nabla_i H_i \cdot \n}_{S_m}&=&-\jump{ \K_j \nabla_j H_j \cdot \n}_{S_m}, \ m=1,\ldots,N_S,\ i,j=\sigma(m) \label{FFfluxbal}\\
H_\D &=&0 \quad \text{on} \ \Gamma_D\\
\K_\D \nabla H_\D \cdot \n_{\Gamma_N}&=&0 \quad \text{on} \ \Gamma_{N} \\
H_i &=&0 \quad \text{on} \ \gamma_{iD}, \ i=1,\ldots,N_F\\
\K_i \nabla_i H_i \cdot \n_{\gamma_{iN}}&=&0 \quad \text{on} \ \gamma_{iN}, \ i=1,\ldots,N_F \label{Strong_neu}
\end{eqnarray} 
where $H_\D$ is the hydraulic head in $\Dt$, $H_i$ the hydraulic head on $F_i$, $i=1,\ldots,N_F$ and $f$ is a volumetric source term. The operator $\nabla$ represents the three-dimensional gradient in $\Dt$, $\nabla_i$ is the two-dimensional gradient on the plane containing fracture $F_i$, whereas $\K_\D(\mathbf{x})\in \mathbb{R}^{3\times3}$, for $\mathbf{x}\subset\Dt$ is a symmetric positive definite matrix representing the transmissivity of the porous matrix and $\K_i(\mathbf{x}) \in \mathbb{R}^{2\times2}$, $\mathbf{x}\subset F_i$ is a symmetric positive definite matrix representing the tangential transmissivity of the fracture $F_i$ on its tangential plane. Finally, $\n_{\Gamma_N}$ is the outward unit normal vector to $\Gamma_N$, and for a given index $i=1,\ldots,N_F$, $\n_{\gamma_{iN}}$ the outward unit normal vector $\gamma_{iN}$ on the plane of fracture $F_i$.

Here, for simplicity, we have considered only the source term on the fractures deriving from the exchange with the porous matrix and homogeneous boundary conditions, but the extension to a more general case is immediate. Conditions~\eqref{MFcont} and \eqref{FFcont} express the continuity of the solution at fracture-matrix interfaces and at fracture intersections, respectively, whereas Equation~\eqref{FFfluxbal} enforces the balance of fluxes at the traces. 

Let us now introduce the following functional spaces: first, on each fracture $F_i$, $i=1,\ldots,N_F$, we define the function space $\Vrm_i$ as $\Vrm_i=\Hrm^1_D(F_i)=\left\lbrace v\in \Hrm^1(F_i): v_{|\gamma_{iD}=0}\right\rbrace$; then on the whole three dimensional domain $\D$,  the space $\Hrm^1_{\Omega}(\D)$ is defined as the space of functions in $\Hrm^1_0(\D)$ whose trace on each interface $\Gamma_i^\pm$ $i=1,\ldots,N_F$ is a function in $\Vrm_i$, i.e.:
$$\Hrm^1_{\Omega}(\D)=\left\lbrace v\in \Hrm^1_0(\D): v_{|\Gamma_{D}=0}, \ v_{|\Gamma_i^\pm}\in \Vrm_i, \ i=1,\ldots, N_F \ \right\rbrace.$$
Also, on each trace $S_m$, $m=1,\ldots,N_S$ we set the spaces $\Ucal^m=\Hrm^{-\frac12}(S_m)$ and $\Hcal^m=\Hrm^{\frac12}(S_m)$. We introduce the following variables: $U_i^m\in \Ucal_m$ defined on trace $S_m$ of fracture $F_i$ as 
\begin{equation}
U_i^m=\jump{\mathbf{K}_i \nabla_i H_i \cdot \n}_{S_m}+ \alpha H_{i|S_m},\quad \forall i=1,\ldots,N_F, \ \forall m\in \mathcal{S}_i,
\label{Uvar}
\end{equation}
thus representing a sort of internal Robin boundary condition on the traces; and, for all $i=1,\ldots, N_F$, $Q_i \in V_i'$, with
\begin{equation}
Q_i:=\jump{\mathbf{K}_\D \nabla H_\D \cdot \n}_{\Gamma_i}+\beta H_{\D|\Gamma_i},
\label{Qvar}
\end{equation}
thus again being a linear combination of the jump of the co-normal derivative of $H_\D$ across interface $\Gamma_i$ and the trace of $H_\D$ on $\Gamma_i$, and $\Vrm_i'$ the dual of $\Vrm_i$. We remark that, as $H_\D\in \Hrm^1_{\Omega}(\D)$ the hydraulic head is continuous across interfaces $\Gamma_i \equiv F_i \subset \D$. 

We also define the bilinear forms: $\aD : \Hrm^1_{\Omega}(\D)\times \Hrm^1_{\Omega}(\D) \mapsto \mathbb{R}$,
$$\aD\left(v, w \right)=\int_\D \mathbf{K}_\D \nabla v \nabla w \, \mathrm{d}\D +\beta \sum_{i=1}^{N_F} \int_{\Gamma_i} v_{|\Gamma_i} w_{|\Gamma_i} \, \mathrm{d}\Gamma;$$
for all $i=1,\ldots,N_F$, bilinear forms $\ai : \Vrm_i\times \Vrm_i \mapsto \mathbb{R}$, 
$$\ai\left(v_i, w_i \right)=\int_{F_i} \mathbf{K}_i \nabla_i v_i \nabla_i w_i \, \mathrm{d}F_i +\alpha \sum_{m\in \mathcal{S}_i} \int_{S_m} v_{i|S_m} w_{i|S_m} \, \mathrm{d}S;$$
$b_i: \Vrm_i'\times \Vrm_i \mapsto \mathbb{R}$
$$b_i\left(q,v\right)=\left\langle q, v\right \rangle_{\Vrm_i',\Vrm_i} $$
and, for $m=1,\ldots,N_S$, form $c^m: \Ucal^m\times \Hcal^m \mapsto \mathbb{R}$,
$$c^m\left(u^m,v\right)=\left\langle u^m,v\right\rangle_{\Ucal^m,\Hcal^m}. $$

Then, problem \eqref{MFcont}-\eqref{Strong_neu} can be written in weak formulation as: find $H_{\D} \in \Hrm^1_{\Omega}(\D)$, $H_i\in \Vrm_i$, $Q_i\in\Vrm_i'$, $U_i^m\in\Ucal_i^m$, $m\in\mathcal{S}_i$, such that, for all $v\in \Hrm^1_{\Omega}(\D)$, for all $w_i\in \Vrm_i$, $i=1,\ldots,N_F$:
\begin{eqnarray}
&&\aD\left(H_\D, v \right)-\sum_{i=1}^{N_F}b_{i}\left(Q_i,v_{|\Gamma_i}\right)=\left\langle f,v\right\rangle_{(\Hrm^1_{\Omega})',\Hrm^1_{\Omega}} \label{STDweak1}\\
&&\ai\left(H_i, w_i \right)-\beta\left( H_{\D|F_i},w_i\right)_{F_i}-\!\!\sum_{m\in\mathcal{S}_i} \!\! c^m\!\left(U_i^m,w_{i|S_m}\right)=-b_i\left(Q_i,w_i\right)\!, \label{STDweak2}
\end{eqnarray}
being $\left(v,w\right)_{\omega}$ the scalar product in $\mathrm{L}^2(\omega)$.
The coupling conditions in weak form are given by: for all  $i=1,\ldots,N_F$, and for all $m=1,\ldots,N_S$
\begin{eqnarray}
b_i\left(H_{\D|F_i}-H_i,\mu_i \right)&=&0, \quad \forall \mu_i\in \Vrm_i',\label{STDweak3}\\
c^m\left(\eta^m,H_{i|S_m}-H_{j|S_m}\right)&=&0, \quad \forall \eta^m\in \Ucal^m, \ {i,j}=\sigma(m) \label{STDweak4}\\
c^m\left(U_i^m+U_j^m,\nu^m\right)&=&0,  \quad \forall \nu^m\in \Hcal^m, \ {i,j}=\sigma(m). \label{STDweak5} 
\end{eqnarray}
Parameters $\alpha>0$ and $\beta>0$ ensure stability of the problems written independently on each fracture and on the three dimensional domain. This is required to obtain a discrete formulation suitable for parallel computing. If $N_F=1$, or equivalently $\mathcal{S}=\emptyset$, well posedness of the problem on the unique fracture $F_1$ is guaranteed by the assumption on $|\gamma_{1D}|>0$.

Problem \eqref{STDweak1}-\eqref{STDweak5} is well posed. To show this, let us introduce the function space $\Hrm^1_{\Omega^+}(\D)$ defined as:
\begin{eqnarray*}
&&\Hrm^1_{\Omega^+}(\D)=\left\lbrace v \in \Hrm^1_0(\D): v_i:=v_{|F_i}\in \Vrm_i, \ \forall i=1,\ldots,N_F,\right.\\ 
&& \hspace{1cm} \left. v_{i|S_m}=v_{j|S_m}, \ \forall m=1,\ldots,N_S, i,j=\sigma(m)\right\rbrace
\end{eqnarray*}
and thus incorporating the matching conditions at the interfaces. Let us then write the following problem: find $H\in\Hrm^1_{\Omega^+}(\D)$ such that, for all $v\in \Hrm^1_{\Omega^+}(\D)$
\begin{equation}
\left(\K \nabla H, \nabla v\right)_{\D} + \sum_{i=1}^{N_F} \left(\K_i \nabla_i H_i, \nabla_i v_i\right)_{F_i}=\left\langle q,v\right\rangle. \label{weakEquiv}
\end{equation}
Problem \eqref{weakEquiv} is well posed, as it can be easily seen that $\Hrm^1_{\Omega^+}(\D)$ is an Hilbert space with the scalar product, \cite{BPSf}:
$$\left(v,w\right)_{\Hrm^1_{\Omega^+}}:=\left(\K \nabla v, \nabla w\right)_{\Dt} + \sum_{i=1}^{N_F} \left(\K_i \nabla_i w_i, \nabla_i v_i\right)_{F_i}.$$
Problem \eqref{STDweak1}-\eqref{STDweak5} is equivalent to problem \eqref{weakEquiv}; indeed, recalling that, for $v\in \Hrm^1_{\Omega^+}(\D)$, conditions \eqref{STDweak3}-\eqref{STDweak4} are satisfied by construction. Moreover summing \eqref{STDweak2} for $i=1,\ldots,N_F$ and \eqref{STDweak1}, using \eqref{STDweak5} and the definition of $U_i^m$ and $Q_i$, for $i=1,\ldots,N_F$, $m\in\mathcal{S}_i$, we get \eqref{weakEquiv}.
We propose a reformulation of problem \eqref{STDweak1}-\eqref{STDweak5} well suited for discretization on non conforming meshes and parallel computing, based on a PDE constrained optimization approach. To this end, we introduce a cost functional expressing the error in the fulfilment of the interface conditions as continuity and flux conservation:
\begin{eqnarray*}
&& J(H_\D,H_\Fcal,U_\mathcal{S}):=\sum_{i=1}^{N_F}\left(\|H_{\D|F_i}-H_i\|^2_{\Vrm_i}\right)+ \nonumber \\ 
&&\hspace{1cm} \sum_{m=1}^{N_S}\left(\|H_{i|S_m}-H_{j|S_m}\|^2_{\Hcal^m}+\|U_i^m+U_j^m-\alpha\left(H_{i|S_m}+H_{j|S_m}\right)\|^2_{\Ucal^m}\right), \label{Functional}
\end{eqnarray*}
being $H_\Fcal:=\prod_{i=1}^{N_F} H_i$ and $U_\mathcal{S}=\prod_{m=1}^{N_S}\prod_{i\in\sigma(m)}U_i^m$.
The solution to problem \eqref{STDweak1}-\eqref{STDweak5} is obtained as the minimum of functional $J(H_\D,H_\Fcal,U_\mathcal{S})$ constrained by the PDE equations on the 3D domain and on the fractures: 
\begin{eqnarray}
\min J(H_\D,H_\Fcal,U_\mathcal{S})&& \\
\text{constrained by}& \text{\eqref{STDweak1}}& \\
\text{and by} &  \text{\eqref{STDweak2}} &\forall i=1,\ldots,N_F.
\end{eqnarray}

\section{Discrete formulation}
\label{DiscrProb}

\begin{table*}
\centering
\caption{Labels used for the dimension of discrete variables} \label{DimTab}
\begin{tabular}{clc}
Label & Description & Definition \\
\hline
$\ND$ & Number of dofs for $h_\D$ & \\
$\NFi$ & Number of dofs for $h_i$ on $\TFi$ & \\
$\NGi$ & Number of dofs for $q_i$ on $\TGi$ &\\
$\NS$ & Number of dofs for $u_i^m$ on $\TS$ & \\
$\Nhf$ & Number of dofs for $h_\Fcal$ & $\sum_{i=1}^{N_F} \NFi$\\
$\NGt$ & Number of dofs for $q$ & $\sum_{i=1}^{N_F} \NGi$\\
$\NSi$ & Number of dofs for $u_i$ & $\sum_{m\in\mathcal{S}_i} \NS$\\
$\NSip$ & Number of dofs for $u^+$ & $\sum_{m\in\mathcal{S}_i}\sum_{k\in\sigma(m)} \NS$\\
$\NSt$ & Number of dofs for $u$ & $\sum_{m=1}^{N_S}\sum_{k\in\sigma(m)} \NS$\\
$\Nht$ & Number of dofs for $h$ &  $\Nhf+ \ND$\\
\hline
\end{tabular}
\end{table*}

The PDE constrained optimization formulation is specifically designed to allow for an easy discretization of the problem using non conforming meshes and to obtain a discrete problem suitable for effective resolution using parallel computing resources. The imposition of the interface constraints expressed by equations \eqref{STDweak3}-\eqref{STDweak5} with a standard approach requires some sort of mesh conformity at the interfaces: either a perfect matching of the nodes on the meshes to enforce conditions by means of degrees of freedom equality constraints, or the weaker condition of alignment of mesh edges with the interfaces, to use mortaring techniques. In contrast, the imposition of interface conditions through the functional only requires the computation of integrals on the traces, as shown below, and thus meshes can be arbitrarily placed with respect to the interfaces, see Figure~\ref{polygonalMesh} for an example of non-conforming meshes in the rock matrix and on the fractures. Further, the minimization process allows to decouple the problems on the fractures and on the three-dimensional domain, for parallel computing.

The discretization strategy proposed in this work is based on the use of standard finite elements on tetrahedra for the three dimensional domain and finite elements on triangles for the fractures. Let us then denote by $\TD$ the tetrahedral mesh on $\D$, characterized by a mesh parameter $\delta_\D$, by $\TFi$ a triangular mesh on $F_i$, $i=1,\ldots,N_F$, with mesh parameter $\delta_{F_i}$, and by $\TGi$ a possibly different triangular mesh on $F_i$, with mesh parameter $\delta_{\Gamma_i}$. We further introduce a discretization of the one-dimensional traces, different on each fracture, denoted by $\TS$, with mesh parameter $\delta_{S_m,i}$, $i=1,\ldots,N_F$, $m\in\mathcal{S}_i$. We denote by $h_\D$ the finite dimensional approximation of $H_\D$ on $\TD$, $h_\D=\sum_{k=1}^{\ND} h_{\D,k} \phi_k$, with $\ND$ the number of degrees of freedom (dofs) and $\phi_k$ a finite element basis function in 3D; for $i=1,\ldots,N_F$, we further call $h_i$ the approximation of $H_i$ on $\TFi$, $h_i=\sum_{k=1}^{\NFi} h_{i,k} \psi_{i,k}$, with $\NFi$ the number of dofs and $\psi_{i,k}$ a 2D basis function; $q_i$ the approximation of $Q_i$ on $\TGi$, $q_i=\sum_{k=1}^{\NGi} q_{i,k}\varphi_{i,k}$ having $\NGi$ dofs, and $\varphi_{i,k}$ one basis function; $u_i^m$ the approximation of $U_i^m$ on $\TS$, $u_i^m=\sum_{k=1}^{\NS} u_{i,k}^m \varrho_{i,k}^m$, with $\NS$ dofs and $\varrho_{i,k}^m$ a basis function. Tables~\ref{DimTab}-\ref{MatrixTab} summarize the labels used for the dimensions of the discrete variables, the name used to denote the basis functions and the notation used, in the following, for the matrices collecting integrals of these basis functions. 
We build arrays of dofs by collecting column-wise the dofs of each discrete function and with abuse of notation we denote the dof array with the same symbol of the corresponding function, thus having arrays $h_\D\in \mathbb{R}^{\ND}$, $h_i\in\mathbb{R}^{\NFi}$, $q_i\in\mathbb{R}^{\NGi}$, $i=1,\ldots,N_F$, and $u_i^m\in\mathbb{R}^{\NS}$, $m\in\mathcal{S}_i$. We define arrays $u^m$, $m=1,\ldots,N_S$, as $u^m=[(u_i^m)^T \ (u_j^m)^T]^T$ for $i,j=\sigma(m)$ with $i<j$, and we further collect column wise arrays $h_i$, $q_i$, $u_i^m$ and $u^m$ forming:
\begin{displaymath}
h_\Fcal=\begin{bmatrix}
h_1\\
\vdots\\
h_{N_F}
\end{bmatrix},
\
q=\begin{bmatrix}
q_1\\
\vdots\\
q_{N_F}
\end{bmatrix},
\
u_i=\begin{bmatrix}
u_i^{m_1}\\
\vdots\\
u_i^{m_{\sharp\mathcal{S}_i}}
\end{bmatrix},
\
u_i^+=\begin{bmatrix}
u^{m_1}\\
\vdots\\
u^{m_{\sharp\mathcal{S}_i}}\\
\end{bmatrix},
\
u=\begin{bmatrix}
u^{1}\\
\vdots\\
u^{m}\\
\vdots\\
u^{N_S}
\end{bmatrix},
\end{displaymath}
where $m_1,\ldots,m_{\sharp\mathcal{S}_i}$ are the indexes in $\mathcal{S}_i$ ordered increasingly.

\begin{table}
\centering
\caption{Summary of nomenclature used for the discrete matrices: involved discrete function names and basis functions} \label{MatrixTab}
\begin{tabular}{cccc}
Matrix letter & function(s) & basis functions & Integration domain\\
\hline
$\Amat_\D$ & $h_\D$ & $\phi$ & $\D$\\
$\Amat_i$ & $h_i$ & $\psi_i$ & $F_i$ \\
$\Bmat_{ij}^m$ & $h_i$,$u_i^m$ & $\psi_i$, $\varrho_j^m$ & $S_m$\\
$\Cmat_{ij}^m$ & $u_i^m$, $u_j^m$ & $\varrho_i^m$, $\varrho_j^m$ & $S_m$\\
$\Dmat_i$ & $h_i$, $q$ & $\psi_i$, $\varphi_i$ & $F_i$ \\
$\Emat_i$ & $h_\D$, $q$ & $\phi$, $\varphi_i$ & $F_i$ \\
$\Gmat_\D^{i}$ & $h_\D$ & $\phi$ & $F_i$ \\
$\Gmat_{F}^{i}$ & $h_i$ & $\psi_i$ & $F_i$ \\
$\Gmat_{\D F}^{i}$ & $h_\D$, $h_i$ & $\phi$, $\psi_i$ & $F_i$ \\
$\Gmat_{ij}^{m}$ & $h_i$, $h_j$ & $\psi_i$, $\psi_j$  & $S_m$ \\
\hline
\end{tabular}
\end{table}

The discrete version of functional $J$ is the following:
\begin{eqnarray}
&&\Jd(h_\D,h_\Fcal,u)=\sum_{i=1}^{N_F}\|h_{\D|F_i}-h_i\|^2_{\Lrm^2(F_i)}+ \nonumber\\
&&\sum_{m=1}^{{N_S}} \left(\|h_i-h_j\|^2_{\Lrm^2(S_m)}+\|u_i^m+u_j^m-\alpha(h_i+h_j)\|^2_{\Lrm^2(S_m)}\right)
\end{eqnarray}
obtained replacing the discretized variables and using $\Lrm^2$ norms. The discrete functional can be written in matrix form, computing the integrals of the basis functions and collecting the values into matrices. Considering the first norm in $\Jd$, we have:
\begin{displaymath}
\|h_{\D|F_i}-h_i\|^2_{\Lrm^2(F_i)}=\int_{F_i}\left( \sum_{k=1}^{\ND} h_{\D,k}\phi_{k|F_i}-\sum_{j=1}^{\NFi}h_{i,j}\psi_{i,j}\right)^2 \, \mathrm{d}F_i 
\end{displaymath}
and we can define three matrices as follows, for each $i=1,\ldots,N_F$, $\Gmat_{\D}^i\in\mathbb{R}^{\ND\times\ND}$, $\Gmat_{\D F}^i\in\mathbb{R}^{\ND\times\NFi}$, $\Gmat_{F}^i\in\mathbb{R}^{\NFi\times\NFi}$:
\begin{displaymath}
(\Gmat_{\D}^i)_{k,\ell}=\int_{F_i} \phi_{k|F_i}\phi_{k|F_i}, \quad (\Gmat_{\D F}^i)_{k,\ell}=\int_{F_i} \phi_{k|F_i}\psi_{i,\ell}, \quad (\Gmat_{F}^i)_{k,\ell}=\int_{F_i} \psi_{i,k}\psi_{i,\ell}
\end{displaymath}
such that 
\begin{displaymath}
\|h_{\D|F_i}-h_i\|^2_{\Lrm^2(F_i)}= 
\begin{bmatrix}
h_\D^T &
h_i^T
\end{bmatrix}
\begin{bmatrix}
\Gmat_{\D}^i & -\Gmat_{\D F}^i\\
-(\Gmat_{\D F}^{i})^T & \Gmat_{F}^i
\end{bmatrix}
\begin{bmatrix}
h_\D\\
h_i
\end{bmatrix}.
\end{displaymath}
The computation of matrix $\Gmat_{\D F}^i$ is not straightforward, as the two involved variables are defined on different meshes. In particular, the intersection of the three dimensional tetrahedral mesh with the fracture plane needs to be computed. This operation defines a polygonal tessellation of $F_i$ which is then sub-triangulated, thus generating a triangular interface mesh. This sub-triangulation process can be performed without any mesh quality requirement, as the resulting mesh is used only for quadrature purposes. The computation of the elements in $\Gmat_{\D F}^i$ is finally performed first computing the intersection of the elements of the interface mesh with the elements in $\TFi$, and subsequently the required integral on the intersection region. Element neighbourhood  information is used to efficiently perform the task. The computation of the interface mesh is a quite complex and expensive task. Also in this case element neighbourhood information is used for efficiency, and further can be performed independently fracture by fracture and thus in parallel, which is of paramount importance for the applicability of the method to complex geometries.
\begin{figure}
\centering
\includegraphics[width=0.95\textwidth]{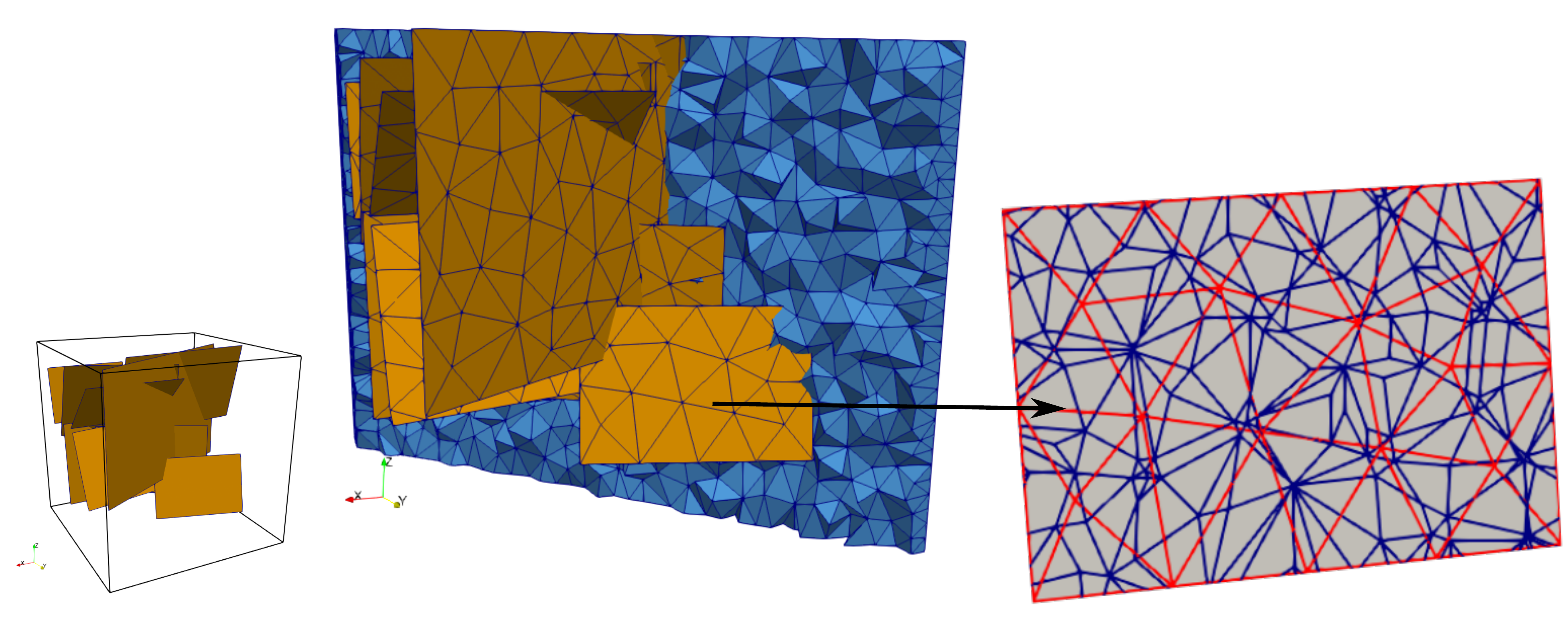}
\caption{Polygonal tessellation on a sample fracture given by the intersection of a tetrahedral mesh with the fracture plane (blue in the right panel) overlapped with the fracture triangular mesh (red in the right panel)}
\label{polygonalMesh}
\end{figure}

We can proceed similarly with the remaining terms of the functional $\Jd$; to this end, for $m=1,\ldots,N_S$,  and all the possible couples of indexes $i,j$ such that $i,j=\sigma(m)$, we define matrices $\Gmat_{ij}^{m}\in\mathbb{R}^{\NFi\times\mathcal{N}_{F_j}}$, $\Bmat_{ij}^{m}\in\mathbb{R}^{\NFi\times\mathcal{N}_j^m}$, $\Cmat_{ij}^{m}\in\mathbb{R}^{\NS\times\mathcal{N}_j^m}$:
\begin{displaymath}
(\Gmat_{ij}^{m})_{k,\ell}=\int_{S_m} \psi_{i,k|S_m}\psi_{j,\ell|S_m}, \ (\Bmat_{ij}^{m})_{k,\ell}=\int_{S_m} \psi_{i,k|S_m}\varrho^m_{j,\ell},  \ (\Cmat_{ij}^{m})_{k,\ell}=\int_{S_m} \varrho^m_{i,k}\varrho^m_{j,\ell}
\end{displaymath}
such that
\begin{displaymath}
\|h_i-h_j\|^2_{\Lrm^2(S_m)}=\begin{bmatrix}
h_i^T & h_j^T
\end{bmatrix}
\begin{bmatrix}
\Gmat_{ii}^{m} & -\Gmat_{ij}^{m} \\
-\Gmat_{ji}^{m} & \Gmat_{jj}^{m}
\end{bmatrix}
\begin{bmatrix}
h_i\\
h_j
\end{bmatrix},
\end{displaymath}
and
\begin{eqnarray*}
&&\|u_i^m+u_j^m-\alpha\left(h_i+h_j\right)\|^2_{\Lrm^2(S_m)}=\\
&&\left(
\begin{bmatrix}
(u_i^m)^T & (u_j^m)^T
\end{bmatrix}
\begin{bmatrix}
\Cmat_{ii}^{m} & \Cmat_{ij}^{m} \\
\Cmat_{ji}^{m} & \Cmat_{jj}^{m}
\end{bmatrix}
-2 \alpha
\begin{bmatrix}
h_i^T & h_j^T
\end{bmatrix}
\begin{bmatrix}
\Bmat_{ii}^{m} & \Bmat_{ij}^{m}\\
\Bmat_{ji}^{m} & \Bmat_{jj}^{m} 
\end{bmatrix}
\right)
\begin{bmatrix}
u_i^m \\
u_j^m
\end{bmatrix}\\
&&+\alpha^2
\begin{bmatrix}
h_i^T & h_j^T
\end{bmatrix}
\begin{bmatrix}
\Gmat_{ii}^{m} & \Gmat_{ij}^{m} \\
\Gmat_{ji}^{m} & \Gmat_{jj}^{m}
\end{bmatrix}
\begin{bmatrix}
h_i \\
h_j
\end{bmatrix}.
\end{eqnarray*}
We can collect these matrices, defined locally at the various interfaces into global matrices to derive a compact form of the functional. Let us define matrix $\Gmat^{\mathcal{S}}\in\mathbb{R}^{\Nhf \times \Nhf}$, $\Nhf=\sum_{i=1}^{N_F} \NFi$, as a $N_F \times N_F$ block matrix, with diagonal blocks in positions $i$-$i$ are given by $(1+\alpha^2)\sum_{m\in\mathcal{S}_i} \Gmat_{ii}^{m}$, $i=1,\ldots,N_F$. Extra diagonal blocks in positions $i$-$j$ ($i\neq j$) are instead equal to $(\alpha^2-1) \Gmat_{ij}^{m}$, if $i,j=\sigma(m)$, or a zero block otherwise. Further let us define matrix $\Gmat_\D^\Fcal=\sum_{i=1}^{N_F} \Gmat_\D^i$, and matrices $\Gmat_{\D F}^\Fcal \in\mathbb{R}^{\ND\times\Nhf}$ and $\Gmat_{F}^\Fcal \in\mathbb{R}^{\Nhf\times\Nhf}$ respectively as
\begin{displaymath}
\Gmat_{\D F}^\Fcal=\begin{bmatrix}
\Gmat_{\D F}^1 \cdots \Gmat_{\D F}^i \cdots \Gmat_{\D F}^{N_F}
\end{bmatrix}, \quad
\Gmat_F^\Fcal=\text{diag}\left(\Gmat_F^1,\ldots,\Gmat_F^i,\ldots,\Gmat_F^{N_F}\right)
\end{displaymath}
Matrix $\Gmat\in\mathbb{R}^{\Nht\times\Nht}$, $\Nht=\Nhf+\ND$, is finally set as
\begin{displaymath}
\Gmat= \begin{bmatrix}
\Gmat_\D^\Fcal & \Gmat_{\D F}^\Fcal \\
(\Gmat_{\D F}^\Fcal)^T & \Gmat_F^\Fcal+\Gmat^\mathcal{S}
\end{bmatrix}.
\end{displaymath}
For all $i=1,\ldots,N_F$, let us assemble matrices $\Bmat_i^+\in\mathbb{R}^{\NFi\times\NSip}$, $\NSip=\sum_{m\in\mathcal{S}_i}\sum_{k\in\sigma(m)} \mathcal{N}_k^m$, collecting row-wise matrices $[\Bmat_{ij}^m \ \Bmat_{ik}^m]$ , for increasing values of $m\in\mathcal{S}_i$ and $j,k=\sigma(m)$, $j<k$, i.e.:
\begin{displaymath}
\Bmat_i^+= \begin{bmatrix}
\Bmat_{ij}^{m_1} & \Bmat_{ik}^{m_1} & \cdots & \Bmat_{ip}^{m_{\sharp\mathcal{S}_i}} & \Bmat_{iz}^{m_{\sharp\mathcal{S}_i}}
\end{bmatrix}
\end{displaymath}
with $p,z=\sigma(m_{\sharp\mathcal{S}_i})$, $p<z$. Let us introduce matrices $\Rmat_i^+\in\mathbb{R}^{\NSip\times\NSt}$, with $\NSt=\sum_{m=1}^{N_S}\sum_{k\in\sigma(m)} \mathcal{N}_k^m$, defined such that $u_i^+=\Rmat_i^+ u$. Matrix $\Bmat^+\in \mathbb{R}^{\Nhf\times \NSt}$ is finally obtained collecting column wise matrices $\Bmat_i^+\Rmat_i^+$.  Matrix $\Cmat \in \mathbb{R}^{\NSt\times \NSt}$, is a block diagonal matrix with $N_S\times N_S$ diagonal blocks, each diagonal block in position $m$-$m$  being equal to
\begin{displaymath}
\begin{bmatrix}
\Cmat_{ii}^{m} & \Cmat_{ij}^{m} \\
\Cmat_{ji}^{m} & \Cmat_{jj}^{m}
\end{bmatrix}, \quad i,j=\sigma(m), \ i<j
\end{displaymath}
The matrix formulation of $\Jd$ then reads as:
\begin{displaymath}
\Jd= \begin{bmatrix}
h_\D^T & h_\Fcal^T
\end{bmatrix}
\Gmat 
\begin{bmatrix}
h_\D\\
 h_\Fcal
\end{bmatrix}
+ u^T \Cmat u -\alpha h_\Fcal^T \Bmat^+ u -\alpha u^T {\Bmat^+}^T h_\Fcal.
\end{displaymath}
We can re-write also the discrete constraint equations in matrix form. We follow a standard procedure and we define matrix $\Amat_\D\in\mathbb{R}^{\ND\times\ND}$ as
\begin{displaymath}
(\Amat_\D)_{k\ell}=\int_\D \K_\D \nabla \phi_k \nabla \phi_\ell +\beta \sum_{i=1}^{N_F} \int_{F_i}  \phi_{k|F_i} \phi_{\ell|F_i},
\end{displaymath}
where the integral on $F_i$ is performed on the interface mesh, generated by the intersection of the tetrahedral mesh with each fracture. Matrices  $\Amat_i\in\mathbb{R}^{\Nhf\times\Nhf}$, for $i=1,\ldots,N_F$ are defined by
\begin{displaymath}
(\Amat_i)_{k\ell}=\int_{F_i} \K_i \nabla_i \psi_{i,k} \nabla_i \psi_{i,\ell} +\alpha \sum_{m\in\mathcal{S}_i} \int_{S_m}  \psi_{i,k|S_m} \psi_{i,\ell |S_m}
\end{displaymath}
which form the diagonal blocks of block diagonal matrix $\Amat_\Fcal\in\mathbb{R}^{\Nhf\times\Nhf}$, $\Amat_\Fcal=\text{diag}(\Amat_1,\ldots,\Amat_{N_F})$.
We introduce, for each fracture $F_i$, $i=1,\ldots,N_F$ matrices $\Rmat_i\in\mathbb{R}^{\NSi \times \NSt}$, $\NSi=\sum_{m\in\mathcal{S}_i} \NS$, defined such that $u_i=\Rmat_i u$, and $\Bmat_i\in\mathbb{R}^{\NFi\times \NSi}$ obtained collecting row-wise matrices $\Bmat_{ii}^m$ for all $m\in\mathcal{S}_i$. These matrices are used for the definition of matrix $\Bmat\in\mathbb{R}^{\Nhf\times\NSt}$, defined grouping column-wise matrices $\Bmat_i$, for $i=1,\ldots,N_F$.
Matrix $\Dmat\in\mathbb{R}^{\Nhf}\times \NGt$ is built as follows:
\begin{displaymath}
\Dmat= \begin{bmatrix}
\Dmat_1 & \cdots & \Dmat_{N_F}
\end{bmatrix}, \quad
(\Dmat_i)_{k\ell}=\int_{F_i} \psi_{i,k}\varphi_{i,\ell}
\end{displaymath}
with integrals computed on the intersection of the mesh $\TGi$ for variable $q_i$ with the mesh $\TFi$ for $h_i$.
We finally introduce matrix $\Emat\in\mathbb{R}^{\ND\times \NGt}$, $\NGt=\sum_{i=1}^{N_F} \NGi$ as 
\begin{displaymath}
\Emat= \begin{bmatrix}
\Emat_{1} & \cdots & \Emat_{N_F}
\end{bmatrix}, \quad
(\Emat_{i})_{k\ell}=\int_{F_i} \phi_{k|F_i} \varphi_{i,\ell}.
\end{displaymath}
where integrals are computed intersecting the mesh $\TGi$ for variable $q_i$ with the triangulated interface mesh given by the intersection between the tetrahedral mesh with the fracture $F_i$. 
 
Setting $h^T=[h_\D^T \ h_\Fcal^T]$, the discrete formulation of the constrained minimization process is:
\begin{eqnarray}
&& \min\left(h^T
\Gmat h + u^T \Cmat u -\alpha h_\Fcal^T \Bmat^+ u -\alpha u^T {\Bmat^+}^T h_\Fcal\right) \label{Pdiscr1}\\
&&\text{constrained by} \nonumber \\
&& \Amat_\D h_\D - \Emat q = b_\D \label{Pdiscr2} \\
&& \Amat_\Fcal h_\Fcal - \beta (\Gmat_{\D F}^\Fcal)^T h_\D - \Bmat u + \Dmat q=0 \label{Pdiscr3} 
\end{eqnarray}
being $b_\D \in\mathbb{R}^{\ND}$ the array resulting from the forcing term.

Let us now introduce the following matrices:
\begin{equation}
\bf{\mathcal{A}}=
\begin{bmatrix}
\Amat_\D & \Omat \\
-\beta (\Gmat^{\D F})^T & \Amat_\Fcal
\end{bmatrix}, \
\bf{\mathcal{B}}=
\begin{bmatrix}
\Emat & \Omat \\
-\Dmat & \Bmat
\end{bmatrix}, \
\bf{\mathcal{B}}^+=
\begin{bmatrix}
\Omat & \Omat \\
\Omat & -\alpha\Bmat^+
\end{bmatrix}, \
\bf{\mathcal{C}}=
\begin{bmatrix}
\Omat & \Omat \\
\Omat & \Cmat
\end{bmatrix}
\label{KKTmatrices}
\end{equation}
and let us collect column-wise variables $q,u$ into variable $w$, then optimality conditions for problem \eqref{Pdiscr1}-\eqref{Pdiscr3} are given by the following linear system:
\begin{equation}
{\bf{\mathcal{M}}}=\begin{bmatrix}
\Gmat & {\bf{\mathcal{B}}}^+ & {\bf{\mathcal{A}}}^T \\
{\bf{\mathcal{B}}^+}^T & \bf{\mathcal{C}} & -{\bf{\mathcal{B}}}^T \\
\bf{\mathcal{A}} & -\bf{\mathcal{B}} & \Omat
\end{bmatrix}, \quad
{\bf{\mathcal{M}}}
\begin{bmatrix}
h\\
w\\
\lambda
\end{bmatrix}=
\begin{bmatrix}
0\\
0\\
b
\end{bmatrix}
\label{KKTsystem}
\end{equation}
with $b^T=[b_\D^T \ \Omat^T]$. Well posedness of problem \eqref{Pdiscr1}-\eqref{Pdiscr3} derives from non singularity of the saddle point matrix ${\bf{\mathcal{M}}}$.

\begin{lemma}
Let matrices $\bf{\mathcal{A}}$, $\bf{\mathcal{B}}$ be defined as in \eqref{KKTmatrices}. Let  $\bf{\mathcal{A}}$ be full rank, let $\bf{\mathcal{L}}=[\bf{\mathcal{A}} \ \ -\bf{\mathcal{B}}]$, and let $\Zmat$ be a matrix obtained collecting row-wise column vectors $z_k$, $k=1,\ldots,\Nw$, $\Nw=\NSt+\NGt$ forming a basis of $\ker(\bf{\mathcal{L}})$, then matrix $\Zmat^T\Gmat \Zmat$ is positive definite. \label{lemma1}
\end{lemma}
\begin{proof}
We start observing that matrix $\bf{\mathcal{A}}$ is full rank as both matrices $\Amat_\D$ and $\Amat_\Fcal$ are full rank under the assumption that $\alpha,\beta > 0$. Then $\dim(\ker({\bf\mathcal{L}}))= \Nw$.
To construct a basis of  $\ker(\mathcal{L})$, let us take  $e_k$, the $k$-th vector of the canonical basis of $\mathbb{R}^{\Nw}$, and let us set $z_k=({\bf{\mathcal{A}}}^{-1}{\bf{\mathcal{B}}} e_k,e_k)$. According to the index $k$, $e_k$ might correspond to a non-null function $q_i$ for some $i=1,\ldots,N_F$ or a non-null function $u_i^m$ for some $i=1,\ldots,N_F$, $m\in\mathcal{S}_i$. In both cases will show that $z_k^T \Gmat z_k > 0$.

\vskip 5pt
\noindent \textit{Non-null function $q_i$}

\noindent Let us start considering the first case and in particular let us assume that $q_i=\varphi_{i,j}$ for a certain index $j=1,\ldots,\NGi$. Let us consider two different scenarios: the case $N_F=1$, i.e. a porous medium with a single fracture and the more general case $N_F>1$. 
\begin{itemize}
\item $N_F=1$\\
As $q_1\neq 0$, for the non singularity of $\Amat_\D$, it is $h_\D\neq 0$, and in particular it can be either $h_\D=\,$const  or $h_\D\neq\,$const.
\begin{itemize}
\item[$\bullet$] $h_\D=\,$const$\, \neq 0$\\
As we assumed homogeneous Dirichlet conditions, $h_\D=\,$const$\,\neq 0$ is possible only if $\Gamma_D=\emptyset$. In virtue of definition \eqref{Qvar}, for the consistency and conformity of the method, we have $\beta h_{\D|F_1}= q_1$, which is possible only if $q_1$ is constant on $F_1$.  By \eqref{STDweak2}, being $|\gamma_{1D}|>0$, it is:
$$a_{F_1}(h_1,\psi_{1,\ell})=-b_1(q_1-\beta h_{\D|F_1},\psi_{1,\ell})=0, \quad \forall \ell=1,\ldots,\mathcal{N}_{h_1}$$
and thus  $h_1=0$, and in particular $\|h_{\D|F_1}-h_1\|_{\mathrm{L}^2(F_1)}>0$.
\item[$\bullet$] $h_\D\neq\,$const\\
Let us set $s_1:=q_1-\beta h_{\D|F_1}$, and let us consider equations~\eqref{STDweak1} and \eqref{STDweak2}, that become:
$$\left(\K_\D \nabla h_\D,\nabla \phi_\ell\right)_\D=b_1(s_1,\phi_{\ell|F_1}), \quad \forall \ell=1,\ldots,\ND$$
$$\left(\K_i \nabla_i h_i,\nabla_i \psi_{1,\ell}\right)_{F_1}=-b_1(s_1,\psi_{1,\ell}), \quad \forall \ell=1,\ldots,\mathcal{N}_{h_1}$$
where the source term $s_1$ appears with opposite sign. We thus have $\|h_{\D|F_1}-h_1\|_{\mathrm{L}^2(F_1)}>0$.
\end{itemize}
\item $N_F>1$\\
If $h_\D=\,$const$\, \neq 0$, proceeding similarly to the case $N_F=1$, we have $\beta h_{\D|F_i}=q_i\neq 0$ and $\|h_{\D|F_i}- h_i\|_{\mathrm{L}^2(F_i)}>0$.\\
If instead $h_\D\neq\,$const, we proceed in the following way: since $q_i\neq 0$ we have $h_i\neq 0$, whereas it is $h_j=0$, for all $j=1,\ldots,N_F$, $j\neq i$. Choosing, in particular, one index $j^\star$ such that fracture $F_i$ and $F_{j^\star}$ intersect in a trace $S_m$, we have $\|h_i-h_{j^\star}\|_{\mathrm{L}^2(S_m)}>0 $.
\end{itemize}

\noindent \textit{Non-null function $u_i^m$}

\noindent Let us now consider $u_i^m=\varrho_{i,j}^m$ for some $i=1,\ldots,N_F$, $m\in\mathcal{S}_i$, and for an index $j=1,\ldots,\NS$, depending on the value of $k$. Also in this case it can be easily shown that we have $h_i\neq 0$, whereas we have $h_\D=0$ and $h_p=0$ for all $p=1,\ldots, N_F$, $p\neq i$, thus having again a non null functional value and thus $z_k^T \Gmat z_k > 0$.

\vskip 5pt
\noindent Being $\Gmat$ positive semi-definite by definition, it is $x^T \Gmat x \geq 0$ and $x^T \Gmat x=0$ if and only if $x \in \ker(\Gmat)$, \cite{HJ} and being  $z_k^T \Gmat z_k > 0$, $z_k \not\in \ker{(\Gmat)}$ for $z=1,\ldots,\Nw$. The space $\mathcal{Z}=\mathrm{span}\{z_1,\ldots,z_{N_w}\}$ is thus a subspace of $\mathrm{Im}(\Gmat)$, and each vector $y\in \mathcal{Z}$ can be written as $y=\Zmat v$, for a vector $v\in \mathbb{R}^{\Nw}$, $v\neq 0$. Then $v^T \Zmat^T \Gmat \Zmat v > 0$.
\end{proof}

\begin{theorem} \label{WPdiscr}
Problem~\eqref{KKTsystem} has a unique solution $h^\star=[(h_\mathcal{D}^\star)^T,(h_\Fcal^\star)^T]^T$, $w^\star=[(q^\star)^T,(u^\star)^T]^T$, $\lambda^\star$, such that $h_\mathcal{D}^\star,h_\Fcal^\star,q^\star,u^\star$ correspond to the constrained minimum of problem \eqref{Pdiscr1}-\eqref{Pdiscr3}.
\end{theorem}
The proof follows from Lemma~\ref{lemma1} applying a classical argument of quadratic programming (see Theorem 16.2 in \cite{NWsecEd}).

\section{Unconstrained optimization problem}
\label{UncProb}
\newcommand{\Acal}{{\bf{\mathcal{A}}}}
\newcommand{\Bcal}{{\bf{\mathcal{B}}}}
\newcommand{\Bpcal}{{\bf{\mathcal{B}}}^+}
\newcommand{\Ccal}{{\bf{\mathcal{C}}}}
\newcommand{\Gcal}{{\bf{\mathcal{G}}}}

We can proceed formally, replacing the constraint equations into the functional, to obtain an unconstrained minimization problem. We have $h=\Acal^{-1}(\Bcal w+b)$, from which we obtain:
\begin{eqnarray*}
\Jd^\star(w)&=&w^T\left(\Bcal^T\Acal^{-T}\Gmat\Acal^{-1}\Bcal + \Ccal - \Bcal^T\Acal^{-T}\Bpcal-\alpha {\Bpcal}^T\Acal^{-1}\Bcal\right)w \\
&+& 2\left(b^T\Acal^{-T}\Gmat\Acal^{-1}\Bcal - b^T\Acal^{-T}\Bcal\right)w + b^T\Acal^{-T}\Gmat\Acal^{-1}b\\
&:=& w^T \Gcal w + 2 g w + \text{const}
\end{eqnarray*}
The unconstrained minimization problem then reads
\begin{equation}
\min_w w^T \Gcal w + 2 g^T w \label{OPTunconstr}
\end{equation}
or equivalently $\Gcal w + g = 0$.
Matrix $\Gcal$ is symmetric positive definite, given the equivalence of \eqref{OPTunconstr} with \eqref{Pdiscr1}-\eqref{Pdiscr3}. The unconstrained minimization problem can thus be solved with a gradient based iterative method, such as the conjugate gradient method.
The steps of the method are as follows:
\begin{itemize}
\setlength\itemsep{-.2em}
\item[] guess $w_0$
\item[] compute $\gamma_0=(\Gcal w_0 + g)$ and set $d_0=-\gamma_0$
\item[] set $k=0$
\item[] while $\gamma_k \neq 0$
\begin{itemize}
\setlength\itemsep{-.1em}
\item[] compute step size $\zeta_k=\frac{\gamma_k^T \gamma_k}{d_k^T\Gcal d_k}$
\item[] set $w_{k+1}=w_k+\zeta_k d_k$
\item[] set $\gamma_{k+1}=\gamma_k+\zeta_k \Gcal d_k$
\item[] compute $\theta_{k+1}=\frac{\gamma_{k+1}^T \gamma_{k+1}}{\gamma_k^T \gamma_k}$
\item[] set $d_{k+1}=-\gamma_{k+1}+\theta_{k+1}d_k$
\item[] set $k=k+1$
\end{itemize}
\item[] end
\end{itemize}
The computation of quantity $y_k=\Gcal d_k$, at each step $k$ can be performed as follows: setting $\bar{h}_k=\Acal^{-1}(\Bcal d_k)$ and $\lambda_k=\Acal^{-T}(\Gmat \bar{h}_k - \Bpcal d_k)$, it is $y_k=\Bcal^T\lambda_k+\Ccal d_k -{\Bpcal}^T \bar{h}_k$.
If $\beta=0$, which is possible as long as there is a non empty portion of the Dirichlet boundary for the three dimensional domain, i.e. $|\Gamma_D|>0$, the computation of $\bar{h}_k$, $\lambda_k$ at each step can be performed independently and in parallel on each fracture and on the three dimensional domain, thus easily allowing to use parallel computing resources for efficient resolution of the scheme, thanks to the block diagonal structure of $\Acal$ and $\Amat_\Fcal$. If $\beta > 0$ then problems on the fractures can be decoupled from the problem in the bulk domain as follows: at step $k>0$, being $\bar{h}_k=[\bar{h}_{\D,k}^T \ \ \bar{h}_{\Fcal,k}^T]^T$ and splitting $d_k=[d_{q,k}^T \ \ d_{u,k}^T]^T$, we compute
\begin{eqnarray*}
\bar{h}_{\D,k}&=&\Amat_\D^{-1}\left(\Emat d_{q,k}\right)\\
\bar{h}_{\Fcal,k}&=&\Amat_\Fcal^{-1}\left(\Bmat d_{u,k}-\beta (\Gmat^{\D F})^T  \bar{h}_{\D,k-1}- \Dmat d_{q,k}\right),
\end{eqnarray*}
and similarly for $\lambda_k$.

\section{Numerical results}
\label{NumRes}
In this section we provide some numerical results in order to show the applicability of the present approach to flow simulations in porous media crossed by arbitrarily complex networks of fractures. All the simulations are performed using linear Lagrangian finite elements on $\TD$ for $h_\D$, linear Lagrangian finite  elements on $\TFi$ for $h_i$, piecewise constant basis functions on $\TGi$ for $q_i$ and piece-wise constant basis functions on $\TS$ for $u_i^m$ on each trace $S_m$, on each fracture $F_i$, $i=1,\ldots,N_F$, $m\in\mathcal{S}_i$.
\subsection{Problems with known solution}
We first propose two simple problems with known analytical solutions, labelled Problem 1 and Problem 2, having the same domain and type of boundary conditions. A cubic domain with unitary edge length is considered; the bottom face is on the plane $z=-\frac12$ with respect to a reference system $\mathcal{O}xyz$, and the cube is crossed by a single fracture $F_1$ placed on the plane $z=0$, see Figure~\ref{P1GeomMesh}, left. The problems are set as follows:
\begin{eqnarray*}
&&\aD\left(H_\D, v \right)-b_1\left(Q_1,v_{|\Gamma_i}\right)=\left(f,v\right)\\
&&a_1\left(H_1, w_1 \right)+\left( H_{\D|F_1},w_1\right)_{F_1}=-b_1\left(Q_1,w_1\right)\\
&&b_1\left(H_{\D|F_1}-H_1,\mu_1 \right)=0, \quad \forall \mu_1\in \Vrm_1'\\
\end{eqnarray*}
with $f=-1$ for Problem 1 and $f=0$ for Problem 2, $\beta=1$, $\K_\D=\K_1=1$ for both problems. Dirichlet boundary conditions are set on cube faces on planes $z=-\frac12$, $z=\frac12$, $x=0$, $x=1$, Neumann boundary conditions on cube faces on planes $y=0$ and $y=1$. Boundary conditions on fracture edges are prescribed accordingly to the boundary conditions on cube edges. Dirichlet and Neumann boundary conditions are derived from the analytical solution, which is $h=\frac14(x^2+y^2)+\frac12|z|$ for Problem 1 and $h=\frac12(x^2-y^2)+z$ for Problem 2. The two problems here considered also share the same meshes. In Figures~\ref{P1GeomMesh}, right and \ref{meshProblem1_1bis}, we display the colormap of the solution of Problem~1. The mesh for the three dimensional domain is non conforming with the fracture plane and independent from the mesh on the fracture, as shown in Figure~\ref{P1GeomMesh}, right. 
In Figure~\ref{P1conv} we report the behaviour of the error with respect to the mesh size both in $L^2$ and in $H^1$ norm for Problem 1.
The three dimensional mesh parameter ranges between $0.02$ and $4\times 10^{-5}$, the mesh on the fracture between $0.3$ and $5\times10^{-3}$. Due to the non conformity of the mesh and to the irregular behaviour of the solution across the interface, sub-optimal convergence trends are obtained. The obtained slopes for the error are compatible with the bounded regularity of the solution $h \not\in H^2(\D)$.
Optimal convergence curves are however recovered if the solution across the interface is smooth. In fact, if we consider Problem 2, having a smooth solution, optimal convergence trends are recovered, as reported in Figure~\ref{P1bisconv}.

\begin{figure}
\begin{minipage}{0.47\textwidth}
\includegraphics[width=0.99\textwidth]{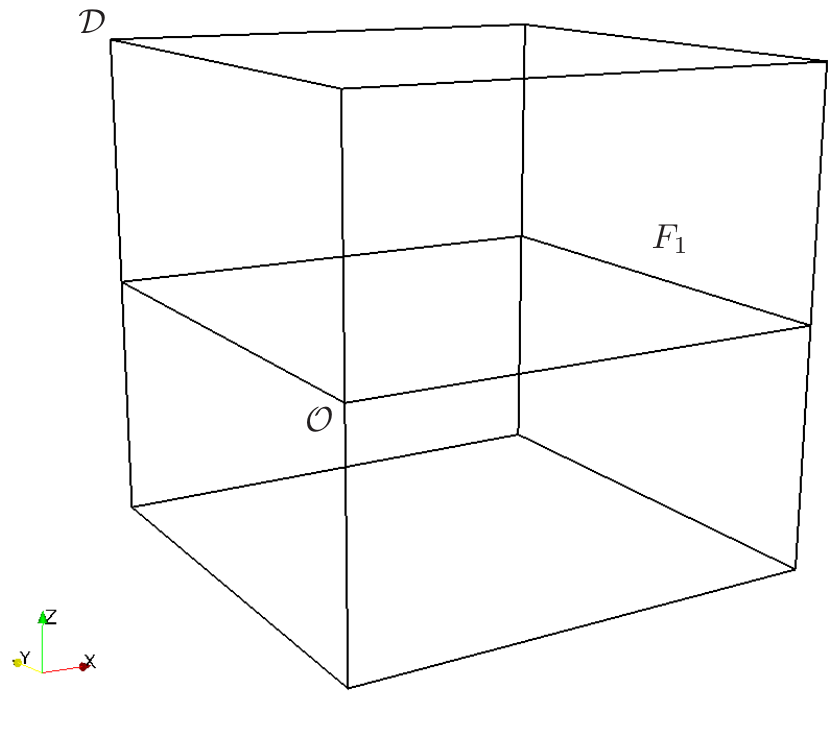}
\end{minipage}
\begin{minipage}{0.47\textwidth}
\includegraphics[width=0.99\textwidth]{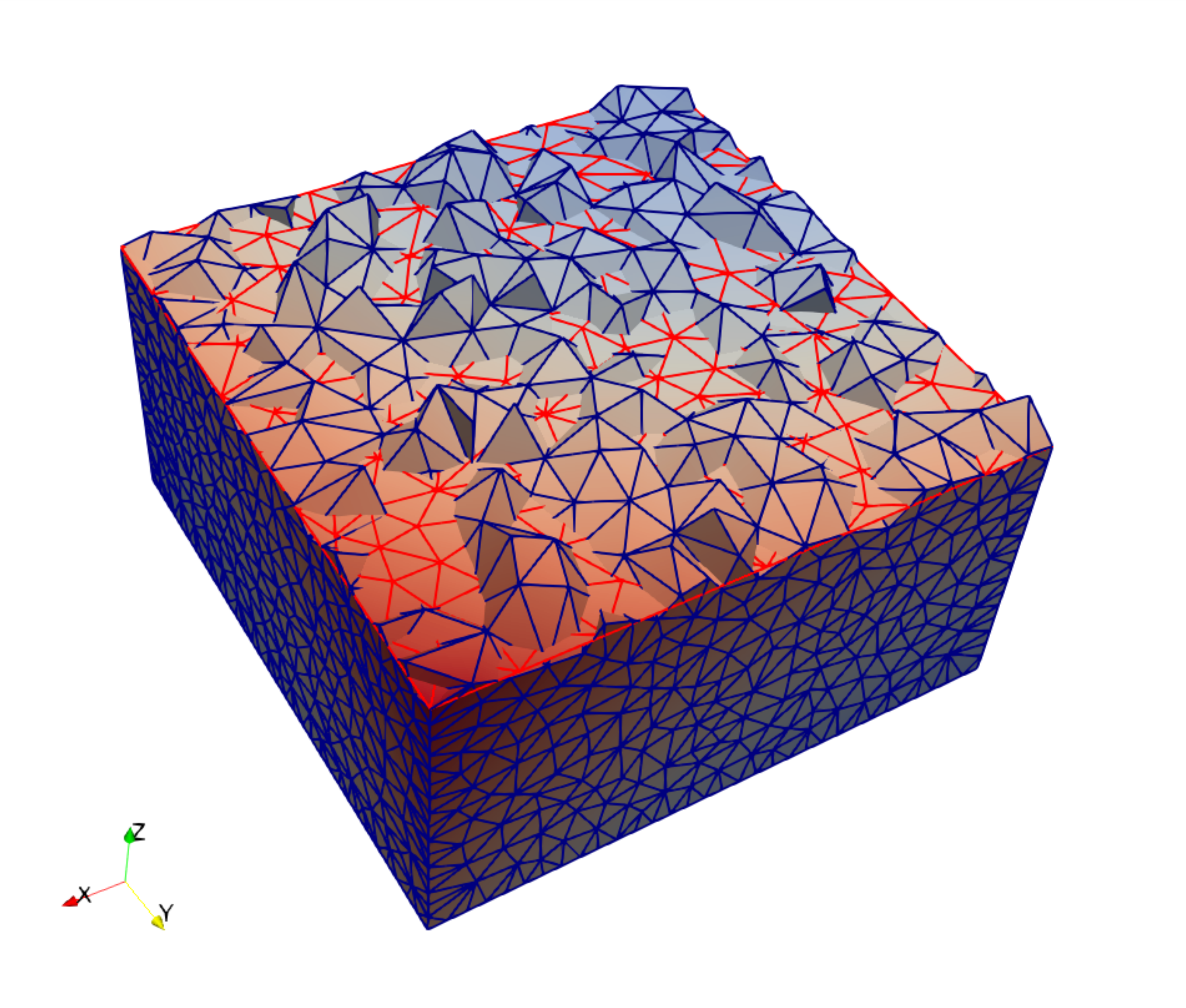}
\end{minipage}
\caption{Problem 1: domain description (left) and detail of the non conforming mesh at the interface (right)}
\label{P1GeomMesh}
\end{figure}
\begin{figure}
\begin{minipage}{0.47\textwidth}
\includegraphics[width=0.99\textwidth]{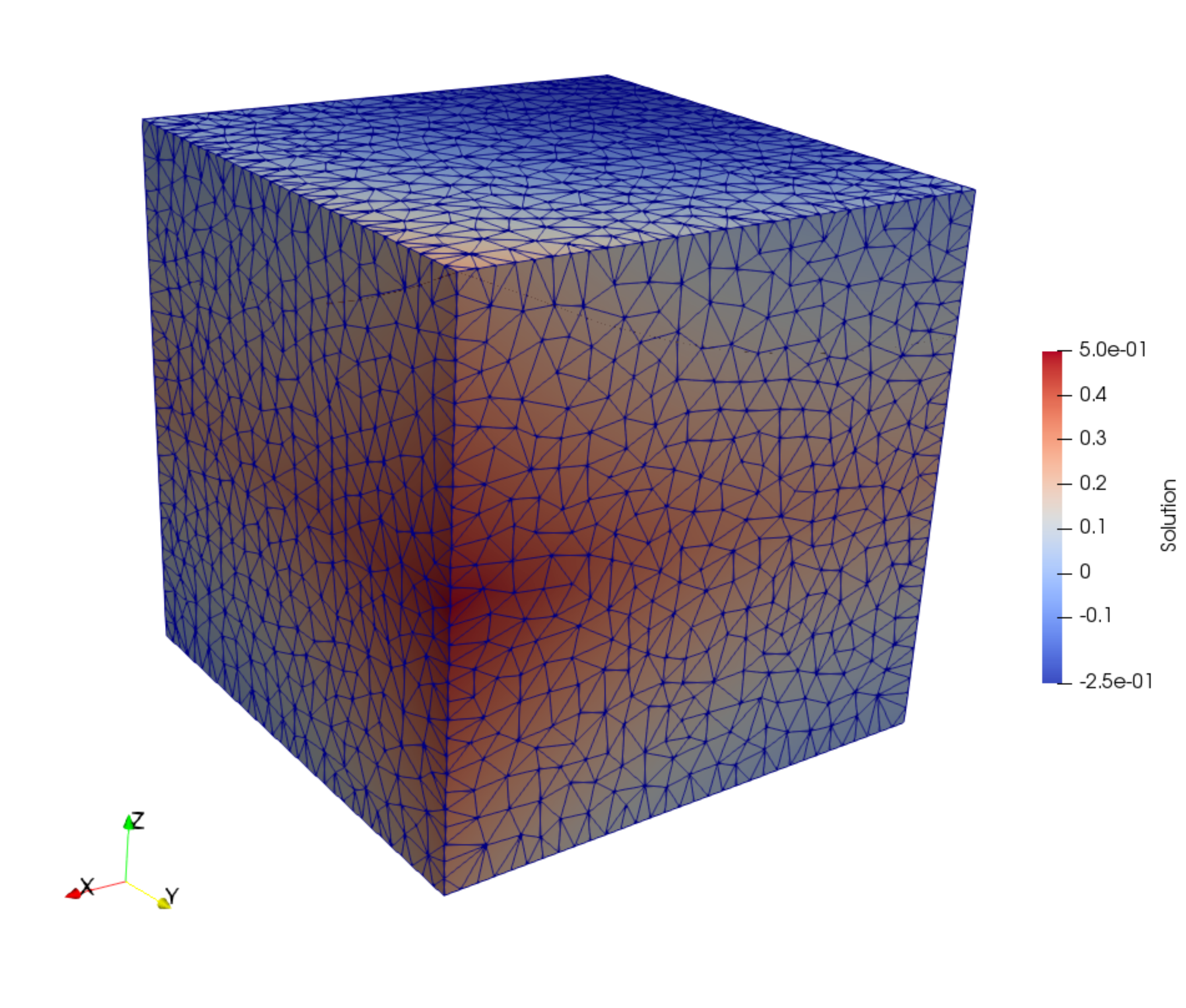}
\end{minipage}
\begin{minipage}{0.47\textwidth}
\includegraphics[width=0.99\textwidth]{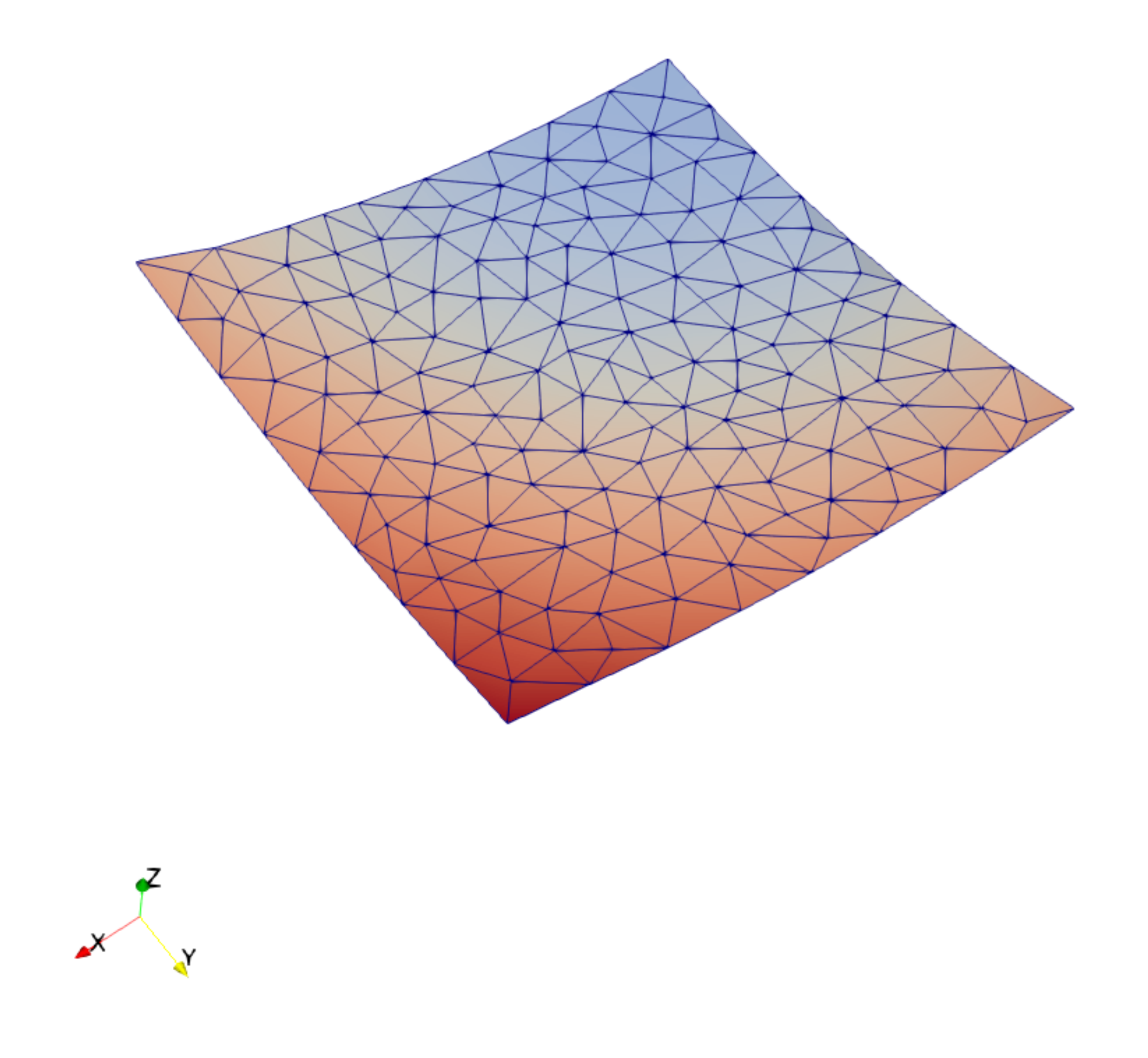}
\end{minipage}
\caption{Problem 1: solution on the three dimensional domain (left) and solution on the fracture (right)}
\label{meshProblem1_1bis}
\end{figure}

\begin{figure}
\begin{minipage}{0.48\textwidth}
\includegraphics[width=0.99\textwidth]{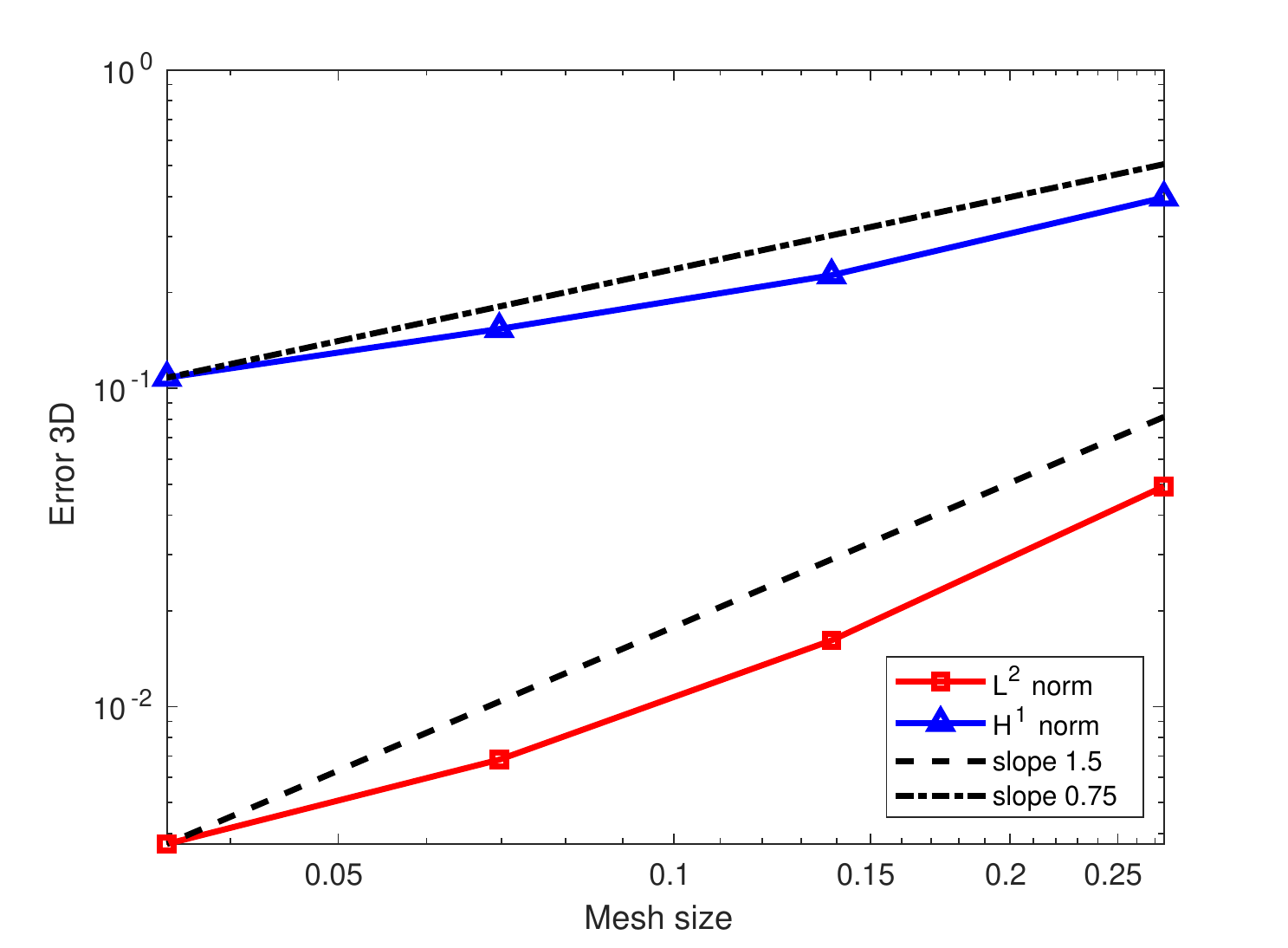}
\end{minipage}
\begin{minipage}{0.48\textwidth}
\includegraphics[width=0.99\textwidth]{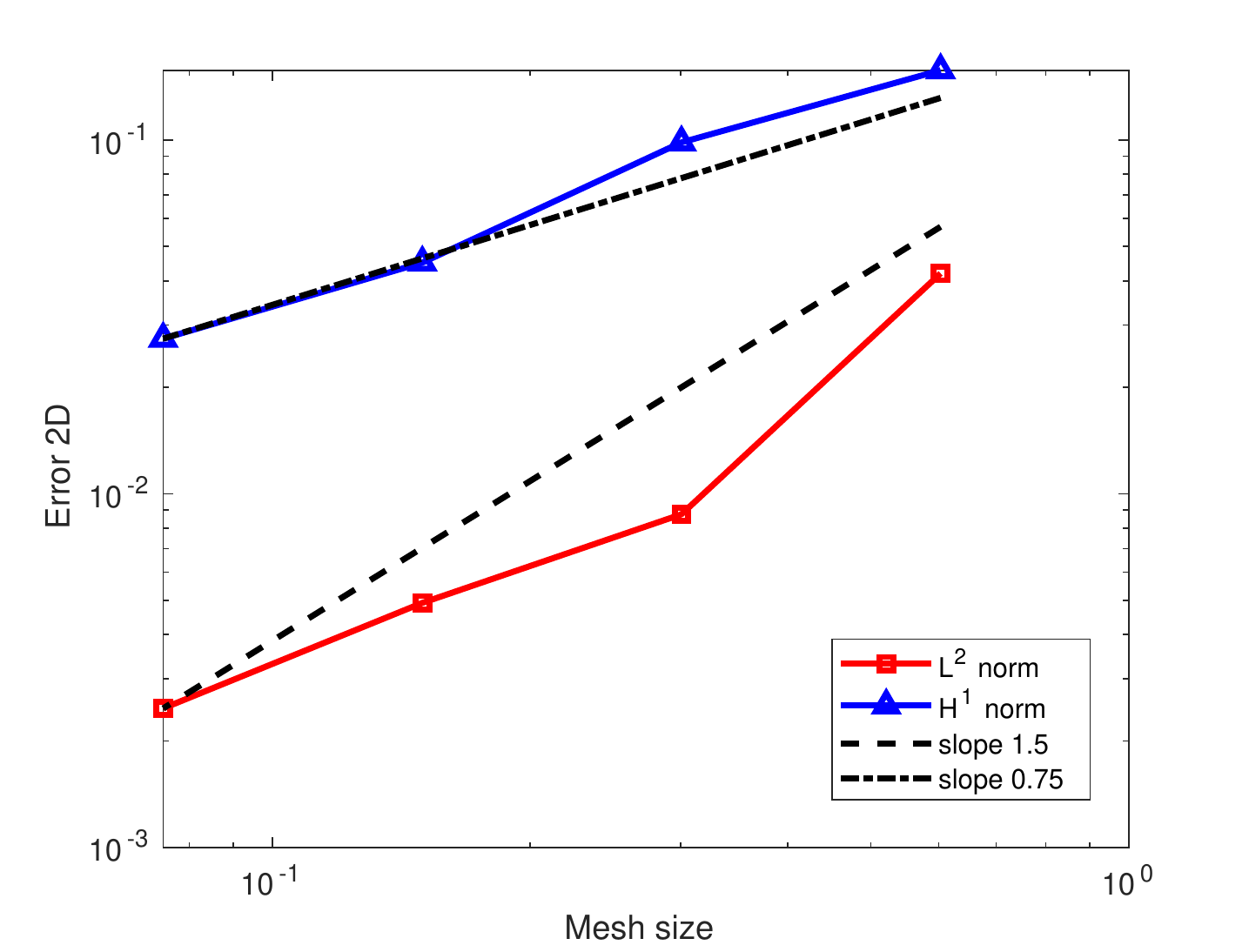}
\end{minipage}
\caption{Problem 1: convergence curves against mesh refinement}
\label{P1conv}
\end{figure}

\begin{figure}
\begin{minipage}{0.48\textwidth}
\includegraphics[width=0.99\textwidth]{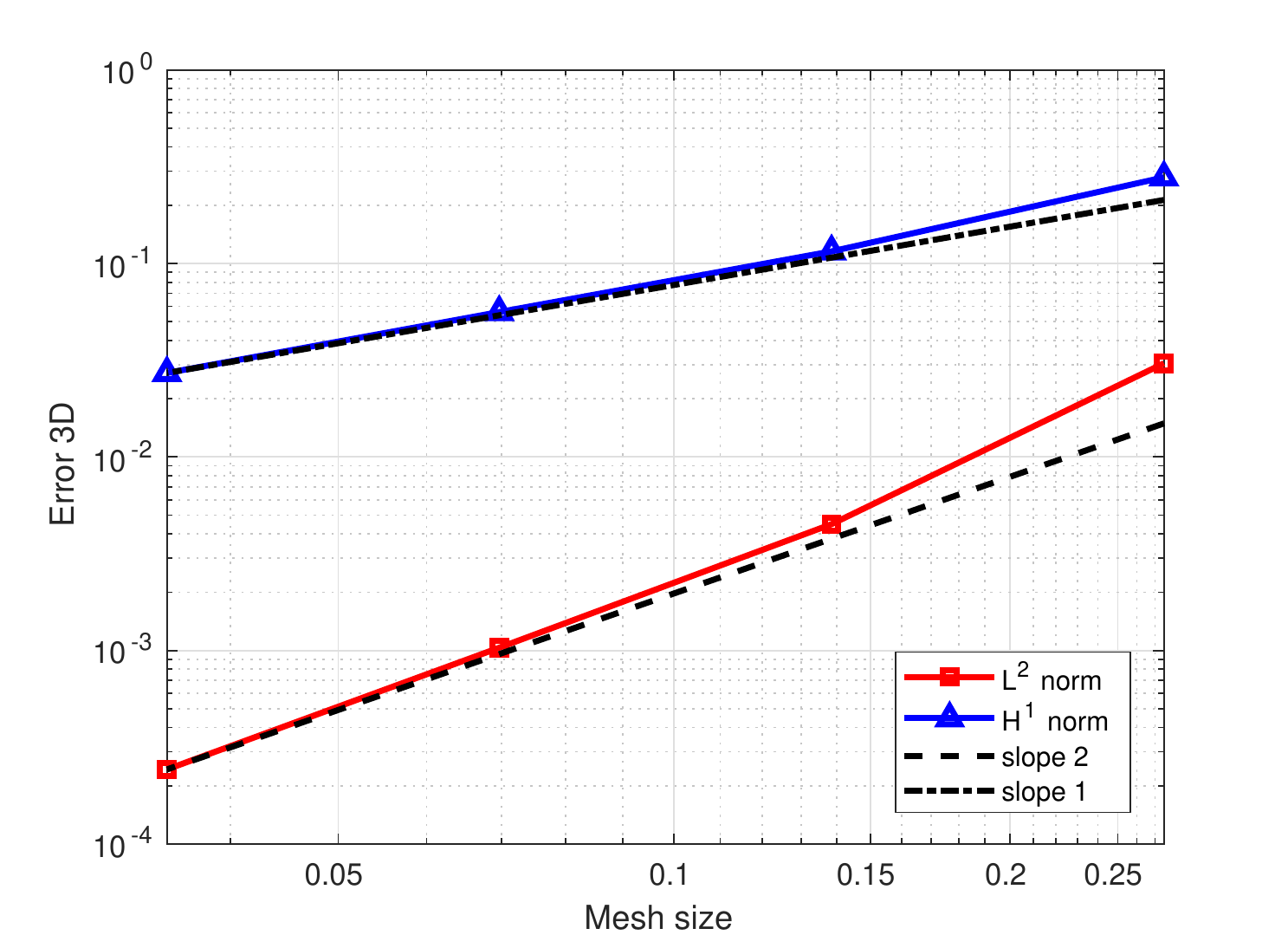}
\end{minipage}
\begin{minipage}{0.48\textwidth}
\includegraphics[width=0.99\textwidth]{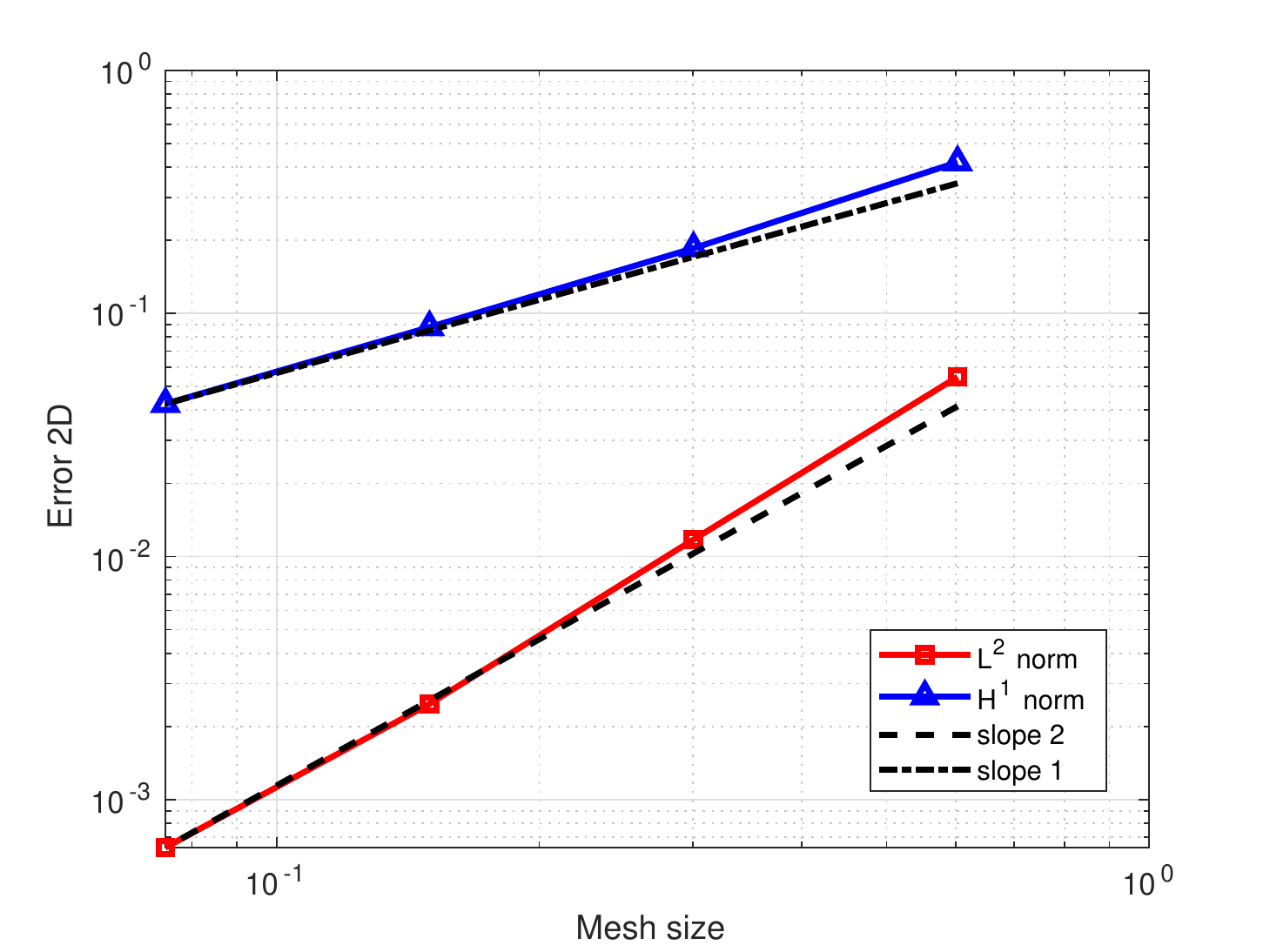}
\end{minipage}
\caption{Problem 2: convergence curves against mesh refinement for a problem with smooth solution across interface}
\label{P1bisconv}
\end{figure}

\subsection{DFN problem}
In the second example, a more complex and realistic DFN is considered, embedded in a cubic domain, with barycentre in the origin of a reference system $\mathcal{O}xyz$ and edge length equal to two, as shown in Figure~\ref{Prob3GeomSol}, on the left. The embedded DFN consists of $20$ randomly placed fractures, forming $62$ traces, with a number of traces per fracture ranging between $4$ and $10$. Traces intersect, forming angles as narrow as $11.5$ degrees, whereas the minimum angle between the normals of couples of intersecting fractures is $17.3$ degrees. A unitary pressure Dirichlet boundary condition is imposed on fracture edges lying on the planes $z=1$, and a zero pressure Dirichlet boundary condition is imposed on the cube face on the plane $z=0$,  all other fracture edges and cube faces being insulated. An inflow is thus obtained through some fracture edges, and outflow occurs through the bottom face of the cube. Figure~\ref{Prob3GeomSol}, on the right, shows the computed solution on the 3D domain and on the fractures, through a section of the three dimensional domain, along with the used mesh, characterized by a mesh parameter equal to $\delta_\D=0.01$ for the tetrahedral mesh and to $\delta_{F_i}=0.05$, $i=1,\ldots,20$ for the triangular mesh on all the fractures. On this mesh, the total number of unknowns, $\Nht+\NGt+\NSt$ is $5541$, and the minimization problem is solved using $90$ iterations to reach a relative residual of $10^{-8}$, resulting in a functional value of $0.069$, on the considered mesh. Considering a refined mesh, with mesh parameter $1.25\times10^{-4}$ for the tetrahedral mesh and to $1.25\times 10^{-2}$ for the triangular mesh, the total number of unknowns rises to $28686$ and the number of iterations to reach the same relative residual is $125$ and the functional value is reduced to $0.057$. We remark that the minimum of the discrete functional is greater than zero, as a consequence of the non conformity of the mesh. These results show the viability of the proposed approach in dealing with complex domains.

\begin{figure}
\begin{minipage}{0.48\textwidth}
\includegraphics[width=0.99\textwidth]{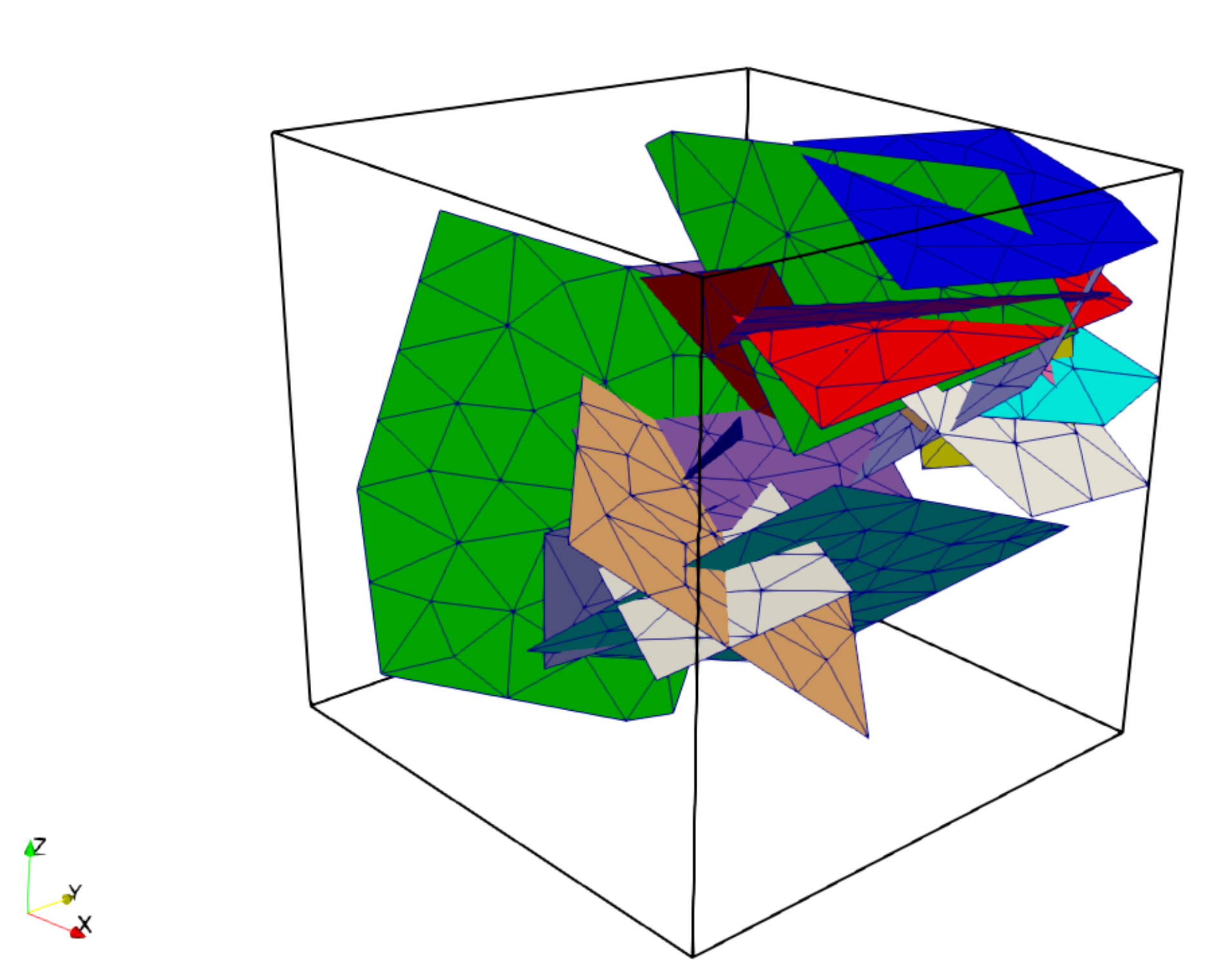}
\end{minipage}
\begin{minipage}{0.48\textwidth}
\includegraphics[width=0.99\textwidth]{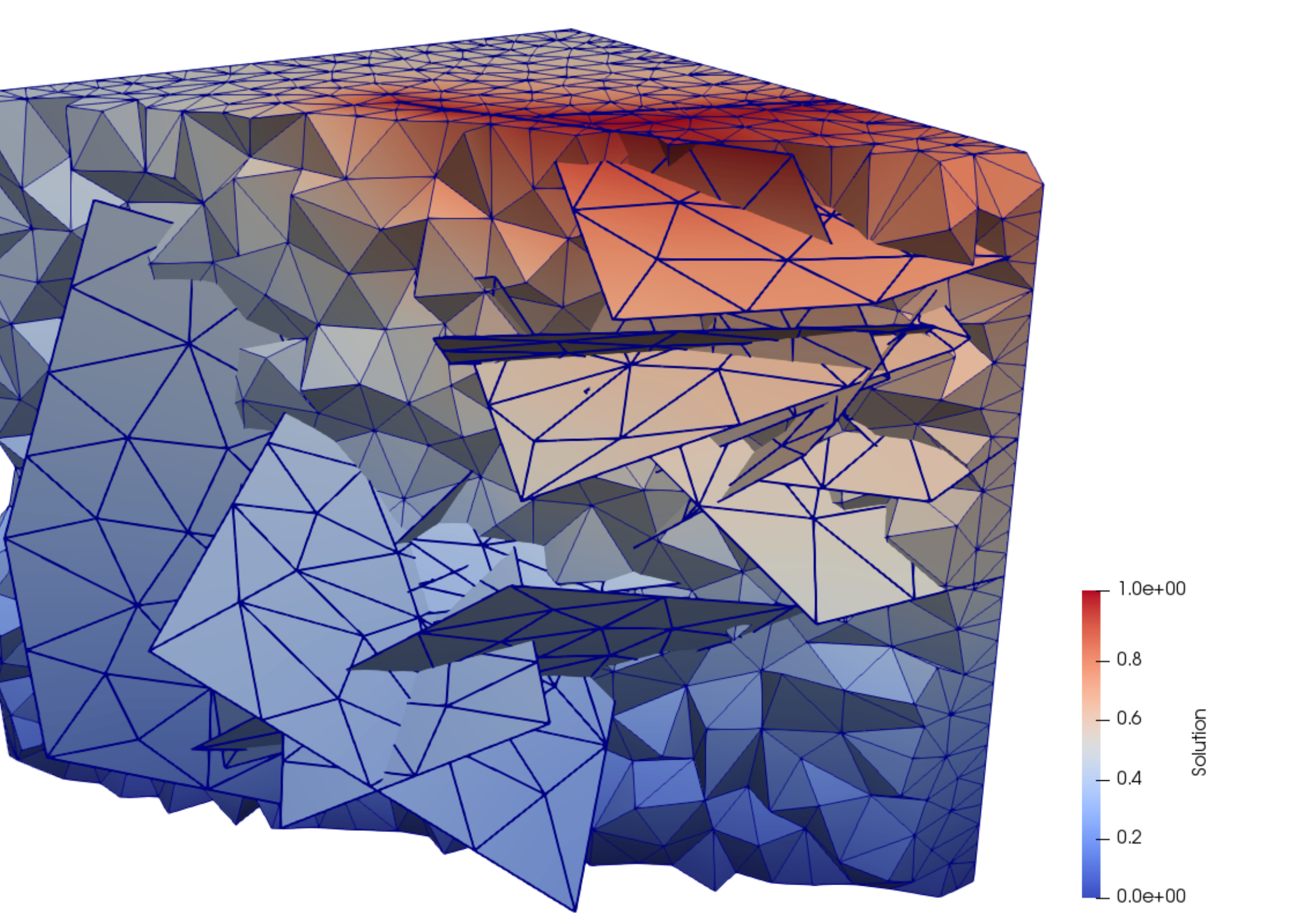}
\end{minipage}
\caption{DFN problem: domain (left) and a section of the computed solution (right)}
\label{Prob3GeomSol}
\end{figure}

\section{Conclusions}
\label{Conc}
A new discretization strategy for the simulation of the flow in arbitrarily complex DFM geometries has been presented and validated. The method is based on standard finite element discretizations for both the three dimensional domain and the fractures, and the meshing can be performed independently on each geometrical entity, thus actually overcoming any mesh related issue for DFM simulations. The resulting discrete problem is well posed and can be efficiently solved via a gradient scheme. The proposed numerical tests validate the method and show its applicability to realistic DFM configurations. Although the proposed method can be easily implemented for parallel solution, optimal parallel solver and suitable well balanced partitioning strategies, yielding to efficient parallel solvers, should be investigated but are out of the scope of the present work.

\begin{acknowledgements}
This work was supported by the MIUR project ``Dipartimenti di Eccellenza 2018-2022'' (CUP E11G18000350001), PRIN project "Virtual Element Methods: Analysis and Applications" (201744KLJL\_004), by INdAM-GNCS and by SmartData@polito.
\end{acknowledgements}

\bibliographystyle{spmpsci}
\bibliography{scico2mp,dfn,misc}

\providecommand{\noopsort}[1]{}
\begin{thebibliography}{10}
\providecommand{\url}[1]{{#1}}
\providecommand{\urlprefix}{URL }
\expandafter\ifx\csname urlstyle\endcsname\relax
  \providecommand{\doi}[1]{DOI~\discretionary{}{}{}#1}\else
  \providecommand{\doi}{DOI~\discretionary{}{}{}\begingroup
  \urlstyle{rm}\Url}\fi

\bibitem{Ahmed2015}
Ahmed, R., Edwards, M., Lamine, S., Huisman, B., Pal, M.: Control-volume
  distributed multi-point flux approximation coupled with a lower-dimensional
  fracture model.
\newblock Journal of Computational Physics \textbf{284}, 462 -- 489 (2015).
\newblock \doi{10.1016/j.jcp.2014.12.047}

\bibitem{Wheeler2015}
Al-Hinai, O., Srinivasan, S., Wheeler, M.F.: Mimetic finite differences for
  flow in fractures from microseismic data.
\newblock In: SPE Reservoir Simulation Symposium. Society of Petroleum
  Engineers (2015)

\bibitem{Angot2009}
{Angot, P.}, {Boyer, F.}, {Hubert, F.}: Asymptotic and numerical modelling of
  flows in fractured porous media.
\newblock ESAIM: M2AN \textbf{43}(2), 239--275 (2009).
\newblock \doi{10.1051/m2an/2008052}

\bibitem{Antonietti2019}
Antonietti, P.F., Facciol\`a, C., Russo, A., Verani, M.: Discontinuous
  {G}alerkin approximation of flows in fractured porous media on polytopic
  grids.
\newblock SIAM Journal on Scientific Computing \textbf{41}(1), A109--A138
  (2019).
\newblock \doi{10.1137/17M1138194}

\bibitem{Benedetto2020}
Benedetto, M.F., Borio, A., Kyburg, F., Mollica, J., Scial\`o, S.: An arbitrary
  order {M}ixed {V}irtual {E}lement formulation for coupled multi-dimensional
  flow problems (2020).
\newblock ArXiv:2001.11309

\bibitem{BBFPSV}
Berrone, S., Borio, A., Fidelibus, C., Pieraccini, S., Scial{\`o}, S., Vicini,
  F.: Advanced computation of steady-state fluid flow in discrete
  fracture-matrix models: {FEM}--{BEM} and {VEM}--{VEM} fracture-block
  coupling.
\newblock GEM - International Journal on Geomathematics \textbf{9}(2), 377--399
  (2018).
\newblock \doi{10.1007/s13137-018-0105-3}

\bibitem{BBoS}
Berrone, S., Borio, A., Scial\`o, S.: A posteriori error estimate for a
  {PDE}-constrained optimization formulation for the flow in {DFN}s.
\newblock SIAM J. Numer. Anal. \textbf{54}(1), 242--261 (2016).
\newblock \doi{10.1137/15M1014760}

\bibitem{BBV2019}
Berrone, S., Borio, A., Vicini, F.: Reliable a posteriori mesh adaptivity in
  discrete fracture network flow simulations.
\newblock Computer Methods in Applied Mechanics and Engineering \textbf{354},
  904 -- 931 (2019).
\newblock \doi{10.1016/j.cma.2019.06.007}

\bibitem{BDV}
Berrone, S., D'Auria, A., Vicini, F.: Fast and robust flow simulations in
  discrete fracture networks with {GPGPUs}.
\newblock GEM - International Journal on Geomathematics \textbf{10}(1), 8
  (2019)

\bibitem{BPSa}
Berrone, S., Pieraccini, S., Scial{\`o}, S.: A {PDE}-constrained optimization
  formulation for discrete fracture network flows.
\newblock SIAM J. Sci. Comput. \textbf{35}(2), B487--B510
  ({\noopsort{2013a}}2013).
\newblock \doi{10.1137/120865884}

\bibitem{BPSb}
Berrone, S., Pieraccini, S., Scial{\`o}, S.: On simulations of discrete
  fracture network flows with an optimization-based extended finite element
  method.
\newblock SIAM J. Sci. Comput. \textbf{35}(2), A908--A935
  ({\noopsort{2013b}}2013).
\newblock \doi{10.1137/120882883}

\bibitem{BPSf}
Berrone, S., Pieraccini, S., Scial\`o, S.: Flow simulations in porous media
  with immersed intersecting fractures.
\newblock Journal of Computational Physics \textbf{345}, 768 -- 791 (2017).
\newblock \doi{10.1016/j.jcp.2017.05.049}

\bibitem{BSV}
Berrone, S., Scial{\`o}, S., Vicini, F.: Parallel meshing, discretization and
  computation of flow in massive {D}iscrete {F}racture {N}etworks.
\newblock SIAM J. Sci. Comput. \textbf{41}(4), C317--C338 (2019).
\newblock \doi{10.1137/18M1228736}

\bibitem{Boon2018}
Boon, W., Nordbotten, J., Yotov, I.: Robust discretization of flow in fractured
  porous media.
\newblock SIAM Journal on Numerical Analysis \textbf{56}(4), 2203--2233 (2018).
\newblock \doi{10.1137/17M1139102}

\bibitem{Brenner2016}
Brenner, K., Groza, M., Guichard, C., Lebeau, G., Masson, R.: Gradient
  discretization of hybrid dimensional darcy flows in fractured porous media.
\newblock Numerische Mathematik \textbf{134}(3), 569--609 (2016).
\newblock \doi{10.1007/s00211-015-0782-x}

\bibitem{Chave2018}
Chave, F.A., Di~Pietro, D.A., Formaggia, L.: {A Hybrid High-Order method for
  Darcy flows in fractured porous media }.
\newblock {SIAM Journal on Scientific Computing} \textbf{40}(2), A1063--A1094
  (2018).
\newblock \doi{10.1137/17M1119500}

\bibitem{DPbook}
Chen, Z., Huan, G., Ma, Y.: {Computational Methods for Multiphase Flows in
  Porous Media}.
\newblock SIAM, Philadelphia, PA, USA (2006)

\bibitem{Chernyshenko2020}
Chernyshenko, A.Y., Olshanskii, M.A.: An unfitted finite element method for the
  darcy problem in a fracture network.
\newblock Journal of Computational and Applied Mathematics \textbf{366}, 112424
  (2020).
\newblock \doi{https://doi.org/10.1016/j.cam.2019.112424}

\bibitem{Coulet2019}
Coulet, J., Faille, I., Girault, V., Guy, N., Nataf, F.: A fully coupled scheme
  using virtual element method and finite volume for poroelasticity.
\newblock Computational Geosciences \textbf{~}, ~ (2019).
\newblock \doi{10.1007/s10596-019-09831-w}

\bibitem{Faille2016}
Faille, I., Fumagalli, A., Jaffr{\'e}, J., Roberts, J.: Model reduction and
  discretization using hybrid finite volumes for flow in porous media
  containing faults.
\newblock Computational Geosciences \textbf{20}(2), 317--339 (2016).
\newblock \doi{10.1007/s10596-016-9558-3}

\bibitem{formaggia2014}
Formaggia, L., Fumagalli, A., Scotti, A., Ruffo, P.: A reduced model for
  {D}arcy’s problem in networks of fractures.
\newblock ESAIM: Mathematical Modelling and Numerical Analysis \textbf{48}(4),
  1089--1116 (2014).
\newblock \doi{10.1051/m2an/2013132}

\bibitem{Fumagalli2019}
Fumagalli, A., Keilegavlen, E.: Dual virtual element methods for discrete
  fracture matrix models.
\newblock Oil Gas Sci. Technol. - Rev. IFP Energies nouvelles \textbf{74}, 41
  (2019).
\newblock \doi{10.2516/ogst/2019008}

\bibitem{2015WR017729}
Fumagalli, A., Pasquale, L., Zonca, S., Micheletti, S.: An upscaling procedure
  for fractured reservoirs with embedded grids.
\newblock Water Resources Research \textbf{52}(8), 6506--6525 (2015).
\newblock \doi{10.1002/2015WR017729}

\bibitem{FS13}
Fumagalli, A., Scotti, A.: A numerical method for two-phase flow in fractured
  porous media with non-matching grids.
\newblock Advances in Water Resources \textbf{62}, 454 -- 464 (2013).
\newblock \doi{10.1016/j.advwatres.2013.04.001}

\bibitem{HJ}
Horn, R.A., Johnson, C.R.: Matrix Analysis.
\newblock Cambridge University Press, Cambridge, United Kingdom (1990)

\bibitem{Li2008}
Li, L., Lee, S.H.: Efficient {F}ield-{S}cale {Simulation} of {B}lack {O}il in a
  {N}aturally {F}ractured {R}eservoir {T}hrough {D}iscrete {F}racture
  {N}etworks and {H}omogenized {M}edia.
\newblock SPE Reservoir Evaluation \& Engineering \textbf{11}(4) (2008).
\newblock \doi{10.2118/103901-PA}

\bibitem{Lipnikov2014}
Lipnikov, K., Manzini, G., Shashkov, M.: Mimetic finite difference method.
\newblock Journal of Computational Physics \textbf{257}, 1163--1227 (2014)

\bibitem{MJR2005}
Martin, V., Jaffr\'{e}, J., Roberts, J.: Modeling fractures and barriers as
  interfaces for flow in porous media.
\newblock SIAM Journal on Scientific Computing \textbf{26}(5), 1667--1691
  (2005).
\newblock \doi{10.1137/S1064827503429363}

\bibitem{Moinfar2014}
Moinfar, A., Varavei, A., Sepehrnoori, K., Johns, R.: {D}evelopment of an
  {E}fficient {E}mbedded {D}iscrete {F}racture {M}odel for 3{D} {C}ompositional
  {R}eservoir {S}imulation in {F}ractured {R}eservoirs.
\newblock SPE Journal \textbf{19}(2), ~ (2014).
\newblock \doi{10.2118/154246-PA}

\bibitem{NWsecEd}
Nocedal, J., Wright, S.J.: Numerical Optimization, {S}econd edn.
\newblock Springer, New York, USA (2006)

\bibitem{Odsaeter2019}
Ods{\ae}ter, L.H., Kvamsdal, T., Larson, M.G.: A simple embedded discrete
  fracture–matrix model for a coupled flow and transport problem in porous
  media.
\newblock Computer Methods in Applied Mechanics and Engineering \textbf{343},
  572 -- 601 (2019).
\newblock \doi{10.1016/j.cma.2018.09.003}

\bibitem{Antonietti2016}
{P. F. Antonietti}, {L. Formaggia}, {A. Scotti}, {M. Verani}, {N. Verzott}:
  Mimetic finite difference approximation of flows in fractured porous media.
\newblock ESAIM: M2AN \textbf{50}(3), 809--832 (2016).
\newblock \doi{10.1051/m2an/2015087}

\bibitem{Qi2005}
Qi, D., Hesketh, T.: An analysis of upscaling techniques for reservoir
  simulation.
\newblock Petroleum Science and Technology \textbf{23}(7-8), 827--842 (2005).
\newblock \doi{10.1081/LFT-200033132}

\bibitem{Sandve2012}
Sandve, T., Berre, I., Nordbotten, J.: An efficient multi-point flux
  approximation method for discrete fracture–matrix simulations.
\newblock Journal of Computational Physics \textbf{231}(9), 3784 -- 3800
  (2012).
\newblock \doi{10.1016/j.jcp.2012.01.023}

\end{thebibliography}

\end{document}